\documentclass{amsart}
\usepackage{tikz-cd}
\usepackage{amsmath}
\usepackage{amsfonts}
\usepackage{cases}
\usepackage{mathrsfs}
\usepackage{bbm}
\usepackage[bookmarks,   colorlinks]{hyperref}
\usepackage[noabbrev,capitalize]{cleveref}
\usepackage{amssymb}
\usepackage{txfonts}
\usepackage{amscd}
\usepackage{amsfonts,    latexsym,   amsmath,   amsthm,   amsxtra,   mathdots}
\usepackage{enumitem}

\usepackage[all,   cmtip]{xy}
\RequirePackage{amsmath} 
\RequirePackage{amssymb}
\usepackage{color}
\usepackage{colordvi}
\usepackage{multicol}
\usepackage{tikz}
\usepackage{graphicx}
\usepackage{amsmath}
\usepackage{amsmath,   amscd}
\usetikzlibrary{matrix,   arrows}
\usepackage{textcomp}
\newcommand{\ie}{\textit{i}.\textit{e}.}
\makeatletter
\DeclareFontFamily{OMX}{MnSymbolE}{}
\DeclareSymbolFont{MnLargeSymbols}{OMX}{MnSymbolE}{m}{n}
\SetSymbolFont{MnLargeSymbols}{bold}{OMX}{MnSymbolE}{b}{n}
\DeclareFontShape{OMX}{MnSymbolE}{m}{n}{
	<-6>  MnSymbolE5
	<6-7> MnSymbolE6
	<7-8>  MnSymbolE7
	<8-9>  MnSymbolE8
	<9-10> MnSymbolE9
	<10-12> MnSymbolE10
	<12->   MnSymbolE12
}{}
\DeclareFontShape{OMX}{MnSymbolE}{b}{n}{
	<-6>   MnSymbolE-Bold5
	<6-7>  MnSymbolE-Bold6
	<7-8>  MnSymbolE-Bold7
	<8-9>  MnSymbolE-Bold8
	<9-10> MnSymbolE-Bold9
	<10-12> MnSymbolE-Bold10
	<12->   MnSymbolE-Bold12
}{}
\let\llangle\@undefined
\let\rrangle\@undefined
\DeclareMathDelimiter{\llangle}{\mathopen}%
{MnLargeSymbols}{'164}{MnLargeSymbols}{'164}
\DeclareMathDelimiter{\rrangle}{\mathclose}%
{MnLargeSymbols}{'171}{MnLargeSymbols}{'171}
\makeatother

\newcommand{\lk}{\operatorname{Lk}}

\hypersetup{
	colorlinks,   
	linkcolor={red!10!black},   
	citecolor={blue!10!black},   
	urlcolor={blue!10!black}
}

\DeclareMathOperator{\Lk}{Lk}

\newtheorem{thm}{Theorem}[section]

\newtheorem{definition}[thm]{Definition} % definition numbers are dependent on theorem numbers
\newtheorem{example}[thm]{Example} % same for example numbers
\newtheorem{lemma}[thm]{Lemma}
\newtheorem{proposition}[thm]{Proposition}
\newtheorem{conjecture}[thm]{Conjecture}
\newtheorem{corollary}[thm]{Corollary}
\newtheorem{remark}[thm]{Remark}

\newtheorem{observation}[thm]{Observation}

\theoremstyle{plain}

\def\qed{\hfill$ \square$ \smallskip}

\DeclareMathOperator{\St}{St}

\DeclareMathOperator{\dlk}{{\Lk}{\downarrow}}
\DeclareMathOperator{\alk}{{\Lk}{\uparrow}}
\DeclareMathOperator{\dst}{{\St}{\downarrow}}

\newcommand{\ra}{\rightarrow}
\newcommand{\s}[1]{$#1 $}
\newcommand{\mn}{\mathbb{N}}
\newcommand{\mz}{\mathbb{Z}}

\newcommand{\xstar}{\{0, 1\}^{<\omega}}

\newcommand{\vdeltag}{V(G)}

\newcommand{\mc}{\mathfrak{C}}

\newcommand{\dy}{\mz[\frac{1}{2}]}

\newcommand{\Rmnum}[1]{\expandafter\@slowromancap\romannumeral #1@}
\newcommand{\xinf}{\{0, 1\}^{\omega}}
\newcommand{\vg}{V_\phi(G)}
\newcommand{\germ}{\mathcal{G}}

\newcommand{\pg}{|\mathcal{P}(G,\varphi)|}
\newcommand{\stein}{X_{G,\phi}}

\newcommand\precdot{\mathrel{\ooalign{$\prec$\cr
  \hidewidth\raise0.225ex\hbox{$\cdot\mkern0.5mu$}\cr}}}

\begin{document}

\title{Embedding groups into boundedly acyclic groups}

\author{Fan Wu}
\address{School of Mathematical Sciences, Fudan University, 220 Handan Road, Shanghai 200433, China}
\email{wuf22@m.fudan.edu.cn}

\author{Xiaolei Wu}
\address{Shanghai Center for Mathematical Sciences, Jiangwan Campus, Fudan University, No.2005 Songhu Road, Shanghai, 200438, P.R. China}
\email{xiaoleiwu@fudan.edu.cn}

\author{Mengfei Zhao}
\address{Shanghai Center for Mathematical Sciences, Jiangwan Campus, Fudan University, No.2005 Songhu Road, Shanghai, 200438, P.R. China}
\email{mfzhao22@m.fudan.edu.cn}

\author{Zixiang Zhou}
\address{School of Mathematical Sciences, Fudan University, Handan Road 220, Shanghai, 200433, China}
\email{zxzhou22@m.fudan.edu.cn}

%    General info

\date{April 2025}


\begin{abstract}
We show that the \s{\phi}-labeled Thompson groups and the twisted Brin--Thompson groups are boundedly acyclic. This allows us to prove several new embedding results for groups.
First,  every group of type $F_n$ embeds quasi-isometrically into a boundedly acyclic group of type $F_n$ that has no proper finite index subgroups. This improves a result of Bridson and a theorem of  Fournier-Facio--L\"oh--Moraschini. Second, every group of type $F_n$ embeds quasi-isometrically into a $5$-uniformly perfect group of type $F_n$. Third, using Belk--Zaremsky's construction of twisted Brin--Thompson groups, we show that every finitely generated group embeds quasi-isometrically into a finitely generated boundedly acyclic simple  group. We also partially answer some questions of Brothier and Tanushevski regarding the finiteness property of $\phi$-labeled Thompson group $V_\phi(G)$ and $F_\phi(G)$. 
\end{abstract}

\maketitle

\section*{Introduction}

The bounded cohomology of groups was first studied by Johnson and Trauber in the seventies in the context of Banach algebras \cite{Jo72}. The theory was then generalized to topological spaces in Gromov's seminal work \cite{Gr82}. Since then, bounded cohomology has found applications in many different fields, see for example \cite{BM02, Ca09, Gy01, MS06} and \cite[Chapter 12]{Fri17}, and has become an independent and active research field. 

Though the bounded cohomology of a group can be defined rather easily (see for example \cite[Chapter 1]{Fri17}), it is hard to compute in general except in lower degrees. Surprisingly, Johnson proved that all amenable groups are boundedly acyclic \cite{Gr82,Jo72}. Recall that a group $G$ is \emph{boundedly acyclic} if its bounded cohomology $\mathrm{H}_b^i(G,\mathbb{R}) = 0$ for any $i\geq 1$. This leads people to consider bounded acyclicity as a weak form of amenability for sometime.  Besides amenable groups,  Matsumoto and Morita proved in \cite{MM85} that the groups of compactly supported homeomorphisms of  Euclidean spaces are also boundedly acyclic. Similar techniques were developed to prove that mitotic groups \cite{Lo17} and binate groups \cite{FFCM23} are boundedly acyclic.  However, none of these non-amenable examples are finitely generated. Combining mitoses and suitable HNN extensions, Fournier-Facio, L\"oh, and Moraschini provided a  way to functorially embed finitely generated groups into  finitely generated boundedly acyclic groups \cite[Theorem 2]{FFCM21}. They also   constructed the first finitely presented non-amenable example \cite[Theorem 4]{FFCM21}. A  breakthrough was then made by Monod \cite{Mo22}, when he showed that the Lamplighter groups (i.e. any group of the form $ G\wr \mathbb{Z}$) and the Thompson group $F$ are boundedly acyclic. Using new ideas in \cite{Mo22,MN23}, Andritsch was able to show that the Thompson group $V$ is also boundedly acyclic \cite{An22}. Note  that the Thompson group $T$ however is not boundedly acyclic \cite[Corollary 6.12]{FFCM23}. More recently,  Campagnolo, Fournier-Facio, Lodha and Moraschini provided further examples of boundedly acyclic groups  in \cite{CFLM23}, including the finitely generated simple orderable groups constructed in \cite{HL19, LBMB22, MBT20}. In the current paper, we add a few new classes of boundedly acyclic groups to the list.

The first class of groups we consider is the labeled Thompson group $V(G)$. This class of groups was first defined by Thompson in \cite{Tho80} (see also \cite[Section 5.5]{DS22}) under the name Splinter group, then rediscovered by Nekrashevych \cite[Section 3.5]{Nek18} and Tanushevski in \cite{Tan14, Tan16}. See also the work of Brothier  \cite{Br21}, and Section \ref{sec:rel-spli} for a discussion on the relations between these groups. Note that Tanushevski only studies the subgroup $F(G)$ of $V(G)$. Brothier on the other hand considers analytical properties (such as the
Haagerup property),  isomorphism problems  among labeled Thompson groups corresponding to some special types of wreath recursions, and descriptions of their automorphism
groups. However, Thompson did prove several interesting results about labeled Thompson groups related to our work: (1) there is an embedding from $G\to V(G)$ that is Frattini and preserves the solvability of the word problem; (2) $V(G)$ is perfect;  (3) if $G$ is finitely generated (resp. finitely presented), so is $V(G)$. 

The definition of  $V(G)$ is  similar to that of the R\"over--Nekrashevych groups, except  now the group $G$ acts trivially on the infinite rooted binary tree. In particular,  $V(G)$ naturally retracts to $V$. Recall that a group is of \emph{type $F_n$} if it admits an aspherical CW-complex with fundamental group $G$ and finite $n$-skeleton, and  a group is called \emph{$N$-uniformly perfect} if every element can be written as the product of $N$ commutators. Uniform perfection implies injectivity of the comparison map $\mathrm{H}^2_b(-,\mathbb{R}) \to \mathrm{H}^2(-,\mathbb{R})$ \cite[Corollary 2.11]{MM85}. We summarize our results in the following theorem.

\begin{thm}\label{thm:labeled-V}
There exists a functor $\mathbf{groups} \to \mathbf{groups}$ associating to each group $G$ the labeled Thompson group $V(G)$ with the following properties:
    \begin{enumerate}
        \item (\Cref{quasi-retract}) There is an injective group homomorphism $\iota_G: G\to V(G)$ that is also a quasi-isometric embedding when $G$ is finitely generated;
        
        \item (\Cref{cor-VG-perfect}) $V(G)$ has no proper finite index subgroups;

        \item (\Cref{thm:unfm-perfect}) $V(G)$ is $5$-uniformly perfect; 
                
        \item\label{mthm-fin} (\Cref{thm:fin-VG}) $V(G)$ is of type $F_n$ if and only if $G$ is of type $F_n$;

        \item (\Cref{Thm:l-Thompson-ba}) $V(G)$ is boundedly acyclic. 
       
    \end{enumerate}
    
\end{thm}

\begin{remark}
 \begin{enumerate}[label=(\roman*)]
    \item (\ref{mthm-fin}) answers a question of Brothier in the  subsection ``Comparison of questions studied" in  \cite[A.2]{Br21}. He further asks the finiteness property of labeled Thompson group corresponding to a general wreath recursion of the form $\phi: G\to G\times G$ for which we only have a partial answer, see \Cref{finiteness-main}. 
 
    \item One could also consider $V(G)$'s subgroup $F(G)$ whose finiteness property has been determined in \cite{Tan16} and \cite[Section 6]{WZ18}. However $F(G)$ is  not boundedly acyclic in general since it retracts to $G$, and it has many finite index subgroups since it surjects to the Thompson group $F$. For a general wreath recursion of the form $\phi: G\to G\times G$, Tanushevski  proves that if $G$ is type $F_1$ or $F_2$, so is $F(G)$ and asks the relation between the finiteness property of $G$ and $F_\phi(G)$ \cite[Introduction]{Tan16}. We improve this to any $n$ in \Cref{thm-fin-FT-G}.

    \item More generally, given any wreath recursion $\phi: G \to G \wr \mathrm{S}_2$, we consider the corresponding $\phi$-labeled Thompson group  $V_{\phi}(G)$ (see Section \ref{Section:L-Thomps}) and prove that it is always boundedly acyclic (Theorem \ref{Thm:l-Thompson-ba}). This includes for example the R\"over--Nekrashevych groups corresponding to the infinite rooted binary tree. See also \Cref{cor:RNGroup-bc} for the  case where the tree is not necessarily binary.

    \item Our proof of type $F_n$ of $G$ implies that of $V_{\phi}(G)$ uses the standard approach by constructing first the corresponding Stein--Farly complex, then applying Brown's criterion. One could potentially prove this by first constructing a cloning system (see Remark \ref{rem:act-closy}), then applying \cite[Proposition 5.9]{WZ18}. This works for $V(G)$ but does not work for a general wreath recursion as the cloning system for $V_{\phi}(G)$ may not be properly graded. For example, if $\phi$ is the left wreath recursion $\phi_l$ in \Cref{ex:examples_of_wreath_recursion} (3), one can check that the corresponding cloning system is not properly graded.

        \end{enumerate}
\end{remark}

We immediately have the following corollaries.

\begin{corollary}\label{cor:emb-bac-fin}
    Every group of type $F_n$ embeds quasi-isometrically into a boundedly acyclic group of type $F_n$ that has no proper finite index subgroups.
\end{corollary}

\begin{remark}
\begin{enumerate}
    \item  Without the boundedly acyclic property, this was established by Bridson \cite{Br98}. Our construction also has the advantage that it is functorial.
    
    \item We also establish a similar embedding result in the setting of $\ell^2$-invisibility, see \Cref{cor-l2-emb}.
\end{enumerate}
   
\end{remark}

\begin{corollary}
    Every group of type $F_n$ embeds quasi-isometrically into a $5$-uniformly perfect group of type $F_n$ that has no proper finite index subgroups.
\end{corollary}

The celebrated Higman's embedding theorem asserts that every recursively presented group embeds into a finitely presented group. Combining this with \Cref{cor:emb-bac-fin}, we have:

\begin{corollary}
    Every recursively presented group embeds into a finitely presented boundedly acyclic group that has no proper finite index subgroups.
\end{corollary}

\begin{remark}
    Without the no proper finite index subgroups property, this has already been  established in \cite[Corollary 5.2]{FFCM21}. On the other hand,  Baumslag, Dyer, Miller proved in \cite[Theorem E]{BDM83} that every recursively presented group embeds into a finitely presented acyclic group.
\end{remark}

Recall that  Collins and Miller constructed in \cite{CM99} a group of type $F$ (it is in fact of geometric dimension $2$) that has unsolvable word problem. Taking the labeled group $G$ to be their group,  we have the following interesting application. 

\begin{corollary}
    There exists a boundedly acyclic group of type $F_\infty$  that has unsolvable word problem and no proper finite index subgroups.
\end{corollary}

Having no proper finite index subgroup is the opposite of being residually finite. And a closely related property is Hopfian. Recall that a group $G$ is called Hopfian if every surjective endormorphism $G\to G$ is an isomorphism. We end the discussion of labeled Thompson group with the following conjecture.

\begin{conjecture}
    Let $G$ be a finitely generated group. Then $V(G)$ is hopfian if and only if $G$ is.
\end{conjecture}
    
\begin{remark}
    Suppose $G$ is not Hopfian and let $f:G\to G$ be a surjection that is not injective. Then the induced map $V(f): V(G)\to V(G)$ is also a surjection that is not injective. This shows if $G$ is not Hopfian, neither is $V(G)$.
\end{remark}

The second class of groups we consider is the twisted Brin--Thompson groups. These groups were first constructed by Belk and Zaremsky in \cite{BZ22} as a generalization of  Brin's high-dimensional Thompson groups \cite{Bri04}. In fact, given any group $G$ acting faithfully on a countable set $S$, they define the associated twisted Brin--Thompson group $\mathrm{SV}_G$ as a subgroup of the group of  homeomorphisms of the Cantor space $\mc^S$. The key feature of twisted Brin--Thompson groups is that they are always simple \cite[Theorem 3.4]{BZ22}.

\begin{thm}[\Cref{cor:tbt-bc}]\label{thm:TBT-ba} 
    The twisted Brin--Thompson groups $\mathrm{SV}_G$ are all boundedly acyclic.
\end{thm}

\begin{remark}
    By taking $G$ to be the Houghton group $\mathcal{H}_n$ and $S$ to be $n$ copies of natural numbers, we get an example of a boundedly acyclic simple group that is of type $F_{n-1}$ but not $F_n$ using \cite[Corollary G]{BZ22}. 
\end{remark}

More generally, we can show that topological full groups acting extremely proximally on the Cantor set are also boundedly acyclic, see \Cref{thm:tofull-ba}.  Combining \Cref{thm:TBT-ba} with \cite[Theorem C]{BZ22}, we have the following.

\begin{corollary}
    Every finitely generated group embeds into a finitely generated (indeed two-generated) boundedly acyclic simple group quasi-isometrically.  
\end{corollary}

At this point, one might hope that all the  $V$-like Thompson groups are  boundedly acyclic. This however is not true,  as Fournier-Facio, Lodha and Zaremsky showed in \cite{FFLZ22} that the braided Thompson group has infinite dimensional second bounded cohomology. It is also worth pointing out that all the boundedly acyclic groups we have found so far are of infinite cohomological dimension. It is a well-known question whether there is a non-amenable boundedly  acyclic group of type $F$. In fact, we do not even have examples of non-amenable and type $F$ groups that we can calculate their bounded-cohomology in infinitely many degrees.

\subsection*{Acknowledgements.}
XW is currently a member of LMNS and is supported by NSFC \\No.12326601. He thanks Jonas Flechsig for discussions related to the finiteness properties of labeled Thompson groups. ZZ is partially supported by the Linghang Foundation. We thank Josiah Oh for proofreading part of the paper, Konstantin Andritsch, Francesco Fournier-Facio for many useful comments on the paper, Volodymyr Nekrashevych for helpful communications on the history of $V_\phi(G)$. We also thank Matt Zaremsky for pointing out references about $V(G)$  after we posted the first version of our paper on the arXiv. In fact, after some communications with him, we realized that the labeled Thompson groups have already been defined and studied by Thompson in \cite{Tho80} under the name Splinter group. We thank the anonymous referee for suggesting numerous corrections, clarifications, and improvements to an earlier draft of this paper. In particular, the newly added Section \ref{sec:l2-inv} is purely their idea.

\section{$\phi$-labeled Thompson groups}\label{Section:L-Thomps}
In this section, we introduce the basics of $\phi$-labeled Thompson groups, see also \cite[Section 3.5]{Nek18} for more information. Most of the material in this section is not original, but we do have some new results in the last few subsections.

%Note that the construction in \cite[Section 3.5]{Nek18} works for general wreath recursions, but we will focus on injective wreath recursions for simplicity.

\subsection{Wreath recursions} \label{subsection:wrinj}
Let  \s{\{0, 1\}^{<\omega}} be the monoid of finite words with letters in \s{\{0, 1\}} where the empty word is the unit  and the multiplication is given by concatenation of words: $(x_1x_2 \cdots x_m)(y_{1}y_{2} \cdots y_n)=x_1x_2 \cdots x_m y_1y_2 \cdots y_n$. 
 So if \s{u\in \{0, 1\}^{<\omega}},  then \s{u0,u1}  are the concatenations of \s{u} by \s{0} and \s{1}. The length of a finite word \s{u} is denoted by \s{l(u)}.
 
The \emph{infinite rooted  binary  tree \s{\mathcal{T}}} is the countably infinite rooted tree such that every vertex has degree $3$ except the \emph{root} which has degree $2$. Note that \s{\{0, 1\}^{<\omega}} can be identified with the vertex set of  \s{\mathcal{T}}, where the root corresponds to the empty word and edges are of the form $\{u,u0\}$ or $\{u,u1\}$ for any word $u$.

An \emph{automorphism} of \s{\mathcal{T}} is a bijection from \s{\{0, 1\}^{<\omega}} to \s{\{0, 1\}^{<\omega}} that fixes the root and preserves the edge relations, hence it also preserves the length of words. 
The group of automorphisms of \s{\mathcal{T}} is denoted by \s{Aut(\mathcal{T})}.

We  identify the Cantor set \s{\mc} with  \s{\{0, 1\}^{\omega}}, the set of (right) infinite words with letters in \s{\{0, 1\}}. Define $\iota: \{0, 1\}^{\omega} \cup \{0, 1\}^{< \omega}\to [0,1]$ by mapping $x_1x_2\cdots$ to $\Sigma_{i=1}^\infty \frac{x_i}{2^i}$ if the words have infinite length, and $x_1x_2\cdots x_n$ to $\Sigma_{i=1}^n \frac{x_i}{2^i}$ for finite words. Restricting to the subset of eventually $0$ infinite words, $\iota$ is injective with image $\mathbb{Z}[\frac{1}{2}]\cap [0,1)$. We often identify the subset of eventually $0$ infinite words in $\{0, 1\}^{\omega}$   with $\mathbb{Z}[\frac{1}{2}]\cap [0,1)$ using $\iota$. Moreover, we have a lexicographical order on $\{0,1\}^\omega$ and $\{0,1\}^{<\omega}$ by declaring $0<1$, and $\iota$ is  order preserving when restricted to  eventually $0$ infinite words. Note that in general $\iota$ is not injective even restricted to the set of finite words.

The concatenation of a finite word by an infinite word is defined by
$
(x_1x_2 \cdots x_m)(y_{1}y_{2} \cdots )=x_1x_2 \cdots x_my_1y_2 \cdots 
$.  Given a finite word \s{u} and a (finite or infinite) word \s{v}, we say that \s{u} is a \emph{prefix} of \s{v}  if \s{v=uw} for some (finite or infinite)  word \s{w}. The \emph{cone} \s{u\xstar} (resp. \s{u\xinf}) of a finite word \s{u} is defined as the subset of words  in \s{\xstar} (resp.  \s{\xinf}) which has  \s{u} as a prefix.

We will  use \s{ \mathrm{S}_n} to denote the symmetric group of $[n]=\{1,\cdots,n\}$ where   $\mathrm{S}_n$ acts on $[n]$ from the right.

\begin{definition}\label{defn-wr-rec}
 A \emph{wreath recursion} of a group \s{G} is a group homomorphism \s{\phi: G\rightarrow G \wr \mathrm{S}_2}, where $ \mathrm{S}_2$ acts on  $G\times  G$ by permuting the coordinates. Let  \s{\phi_S: G\ra  \mathrm{S}_2} be the composition of \s{\phi} with the natural projection \s{G \wr \mathrm{S}_2\ra  \mathrm{S}_2}. When $\phi$ is injective, we shall call it an \emph{injective wreath recursion}. 
\end{definition}
 %To make our life easier, we will focus  on injective wreath recursions, i.e. when $\phi$ is injective. 

Given a  group $G$, there are some canonical  wreath recursions one can consider: 
\begin{example}\label{ex:examples_of_wreath_recursion}\ \ 
\begin{enumerate}
        \item The diagonal  wreath recursion: \s{\Delta:g\ra((g,g),\mathrm{id}_{\mathrm{S}_2})};
        \item The vanishing wreath recursion: $\nu: g \to ((1,1),\mathrm{id}_{\mathrm{S}_2})$;
        \item  The right  wreath recursion: \s{\phi_r: g\ra{((1, g),\mathrm{id}_{\mathrm{S}_2})}}; similarly, the left wreath recursion \s{\phi_l: g\ra{((g, 1),\mathrm{id}_{\mathrm{S}_2})}}; 
        \item \s{\phi_{\kappa} : g\ra ((g, g),\kappa(g))} where  \s{\kappa} is a homomorphism from \s{G} to \s{ \mathrm{S}_2}.
    \end{enumerate}   
\end{example}

Note that \s{Aut(\mathcal{T})}  has a canonical injective  wreath recursion structure: 
the action of \s{Aut(\mathcal{T})} on length-one words induces a homomorphism
\s{\rho: Aut(\mathcal{T})\ra  \mathrm{S}_2}. Every element in \s{ker(\rho)} can be regarded as an automorphism of the subtrees spanned by 
\s{0\xstar} and $ 1\xstar$, hence \s{ker(\rho)}
is isomorphic to the direct product of two copies of \s{Aut(\mathcal{T})} which \s{ \mathrm{S}_2} acts on by permuting the coordinates. This gives an isomorphism \s{\psi_T: Aut(\mathcal{T})\cong  Aut(\mathcal{T})\wr \mathrm{S}_2.}

   There is a natural (right) action of \s{G} on the rooted binary infinite tree \s{\mathcal{T}} with respect to \s{\phi}. 
The action of \s{G} on \s{\mathcal{T}} is defined by induction on the length of words:  

\begin{enumerate}
    \item \s{G} preserves the empty word.
    \item Suppose that the action of words with length at most \s{n} is defined. Then 
for \s{g \in G,  v=xu} where \s{l(u)=n, x \in \{0, 1\} }, and \s{\phi(g)=((g_0, g_1),\sigma_g)\in G \wr \mathrm{S}_2},  we define \s{v\cdot g :=(x\cdot \sigma_g)(u\cdot g_x)}. 
\end{enumerate}

Therefore a wreath recursion \s{\phi} gives a natural group homomorphism \s{\phi^*} from \s{G} to \s{Aut(\mathcal{T})} such that the following diagram commutes: 
\begin{equation}
\begin{tikzcd}\label{diagram1}
G \arrow{r}{\phi} \arrow[swap]{d}{\phi^\ast} & G \wr \mathrm{S}_2 \arrow{d}{(\phi^*\times \phi^*)\rtimes \mathrm{id}_{ \mathrm{S}_2}} \\
Aut(\mathcal{T}) \arrow{r}{\psi_T}[swap]{\cong} &  Aut(\mathcal{T})\wr \mathrm{S}_2
\end{tikzcd}    
\end{equation}

\begin{remark}
    When \s{\phi^\ast} is injective, \ie,  the action is faithful, \s{G} is called a \emph{self-similar group}, see \cite{Nek05} for more information. Note that the injectivity of $\phi$ is not enough to guarantee the injectivity of $\phi^\ast$. In fact, the diagonal $\Delta$ wreath recursion is injective, but the induced action $\phi^\ast$ is trivial. 
\end{remark}

\subsection{ \s{G}-matrix and the  \s{\phi}-labeled Thompson group. }

Given a wreath recursion \s{\phi} of a group $G$, we proceed to define the \s{\phi}-labeled Thompson group $V_\phi(G)$.

We call a finite subset \s{\Sigma\subset\xstar}  a  {\em partition set}, if one of the following conditions holds (one easily checks these conditions are equivalent): 
\begin{enumerate}
    \item \s{u\{0, 1\}^{<\omega}\cap v\{0, 1\}^{<\omega}=\varnothing} for any $u\neq v$ in $ \Sigma$ and the set
\s{\{0, 1\}^{<\omega}-\bigcup_{u\in \Sigma}u\{0, 1\}^{<\omega}} is finite. 

\item Any infinite word has a unique element in \s{\Sigma} as its prefix. 

\item There exists a number \s{n} such that any finite word of length greater than \s{n} has an unique element in \s{\Sigma} as its prefix. 
\end{enumerate}

 Every partition set \s{\Sigma} gives a partition \s{\{u\{0,1\}^\omega:u\in \Sigma\}} of   $\{0,1\}^\omega$  which we can identify with the boundary \s{\partial\mathcal{T}}  of \s{\mathcal{T}}. Note that it  also gives to a partition of the subset of eventually $0$ infinite words (recall we identify it  with \s{\dy \cap [0,1)}). A partition of \s{\partial\mathcal{T}} or \s{\dy \cap [0,1)} arising from a partition set is called  a \emph{dyadic partition}. The following lemma is immediate from the definition.

\begin{lemma}\label{incomparable words}
Let  \s{u_1,  \cdots , u_k\in \{0, 1\}^{<\omega}} be distinct finite words. Then the cones \s{\{u_i\xstar\}_{i=1}^k} are pairwise disjoint if and only if there is a partition set containing $\{u_1,  \cdots , u_k\}$ as a subset. 
\end{lemma}

From now on, we will list the elements of a partition set using lexicographic order unless otherwise stated. 

\begin{definition}
    Given \s{\sigma\in  \mathrm{S}_n} and partitions \s{\Sigma_1=\{u_1,  \cdots , u_n\},\Sigma_2=\{v_1,  \cdots , v_n\}}, we can define a homeomorphism of  $\{0,1\}^\omega$:
$$f^\sigma_{\Sigma_1,\Sigma_2}:\xinf\ra\xinf, u_iw\mapsto v_{\sigma(i)}w.$$
where $w$ is any infinite word. The subgroup of such homeomorphisms of $\{0,1\}^\omega$ is called \emph{the Thompson group \s{V}}. 
\end{definition}

%\begin{remark}
%    Unless there are extra assumptions, we will assume the elements of $\Sigma_1$ and $\Sigma_2$ are listed using lexicographic order. 
%\end{remark}

For a partition set \s{\Sigma},  if we replace one of its element \s{u} by \s{\{u0, u1\}},  we get a new partition set \s{\Sigma_u}, called {\em the simple expansion} of \s{\Sigma} at \s{u}. 
We say a partition set $\Sigma_1$ is an {\em expansion} of another partition set $\Sigma_2$ if $\Sigma_1$ can be obtained from $\Sigma_2$ by  finitely many  simple expansions. 

We proceed now to define the $\phi$-labeled Thompson group $V_\phi(G)$. A \emph{\s{G}-matrix} is a block of the following form:
$
\begin{scriptsize}\begin{pmatrix}  
u_1 & u_2 & \cdots & u_n \\
g_1 & g_2 & \cdots  &g_n \\
v_1&v_2 & \cdots  &v_n
\end{pmatrix}\end{scriptsize}  $
where \s{\{u_1,  \cdots , u_n\}}, \s{\{v_1,  \cdots , v_n\}} are two partition sets and \s{g_1,  \cdots , g_n\in G}. \s{\{u_1,  \cdots , u_n\}} is called the \emph{domain}  and \s{\{v_1,  \cdots , v_n\}} is called the \emph{range} of the matrix.  We shall consider two \s{G}-matrices to be the same,  if one is obtained from the other by permuting the columns.  Note that we can always write a \s{G}-matrix in the form \s{
\begin{scriptsize}\begin{pmatrix}  
u_1 & u_2 & \cdots & u_n \\
g_1 & g_2 & \cdots  &g_n \\
v_1&v_2 & \cdots  &v_n
\end{pmatrix}\end{scriptsize}  
} where \s{u_1<\cdots <u_n} (here we are using the lexicographic order on $\{0,1\}^{<\omega}$). Note that the order of \s{\{v_1,  \cdots , v_n\}} further gives  a permutation \s{\sigma}, by mapping $i$ to the position of $v_i$ after we order them from small to large.

\begin{definition}[Expansion rule]
  Let \s{\phi: G\ra G \wr \mathrm{S}_2} be a wreath recursion. 
 The \emph{simple expansion} of the \s{G}-matrix \s{\begin{scriptsize}\begin{pmatrix}    
u_1 & u_2 & \cdots &u_k& \cdots & u_n \\
g_1 & g_2 & \cdots  &g_k& \cdots &g_n \\
v_1&v_2 & \cdots &v_k& \cdots &v_n
\end{pmatrix}\end{scriptsize}  }
at \s{u_k} is defined by: 

\begin{center}
 \s{\begin{scriptsize}\begin{pmatrix}  
u_1 & u_2 & \cdots &u_k0&u_k1& \cdots & u_n \\
g_1 & g_2 & \cdots  &g_{k, 0}&g_{k, 1}& \cdots &g_n \\
v_1&v_2 & \cdots &v_k\cdot (0)\sigma_{g_k} &v_k\cdot (1)\sigma_{g_k} & \cdots &v_n
\end{pmatrix}\end{scriptsize}  }   
\end{center}
where $((g_{k, 0}, g_{k, 1}),\sigma_{g_k}) = \phi(g_k)$. 
\end{definition}

    Given two \s{G}-matrices \s{\alpha, \beta}, we say \s{\beta}  is an \emph{expansion} of \s{\alpha} if it can be obtained by a finite sequence of simple expansions from $\alpha$. The inverse process of an expansion is called a \emph{reduction}.  We say that two \s{G}-matrices are \emph{equivalent} if they have a common expansion.    %By saying a \s{G}-table we mean a equivalent class of \s{G}-matrices.  
    The equivalence class of \s{\begin{scriptsize}\begin{pmatrix}  
u_1 & u_2 & \cdots & u_n \\
g_1 & g_2 & \cdots &g_n \\
v_1&v_2 & \cdots &v_n
\end{pmatrix}\end{scriptsize}  } is denoted by
\s{  \begin{scriptsize}\begin{bmatrix}  
u_1 & u_2 & \cdots & u_n \\
g_1 & g_2 & \cdots &g_n \\
v_1&v_2 & \cdots &v_n
\end{bmatrix}  \end{scriptsize}  }.

Let \s{V_\phi(G)} be the set of equivalence classes of $G$-matrices with respect to \s{\phi}.  
We now  define a multiplication on \s{V_\phi(G)} to make it  a group: For \s{a, b\in V_\phi(G)}, up to finite expansions, we can assume that the two \s{G}-matrices representing \s{a, b} have the following forms: \s{a=  \begin{scriptsize}\begin{bmatrix}  
u_1 & u_2 & \cdots & u_n \\
g_1 & g_2 & \cdots  &g_n \\
v_1&v_2 & \cdots  &v_n
\end{bmatrix}  \end{scriptsize}} and \s{b=  \begin{scriptsize}\begin{bmatrix}  
v_1 & v_2 & \cdots & v_n \\
h_1 & h_2 & \cdots  &h_n \\
w_1&w_2 & \cdots  &w_n
\end{bmatrix}  \end{scriptsize}}, now define \s{ab=  \begin{scriptsize}\begin{bmatrix}  
u_1 & u_2 & \cdots & u_n \\
g_1h_1 & g_2h_2 & \cdots  &g_nh_n \\
w_1&w_2 & \cdots  &w_n
\end{bmatrix}  \end{scriptsize}}. 
\begin{lemma}\label{vg-well-defined}
    The multiplication given above is well-defined, and makes \s{V_\phi(G)} a group.
\end{lemma}
\begin{proof}
We first show that the multiplication is well-defined. Note first that any two representatives of a $G$-matrix have a common expansion. Thus to prove the multiplication is well-defined, it suffices to prove that if two $G$-matrices $\alpha,\beta$ can perform multiplication already, and their expansions $\alpha',\beta'$ can also do so, then $[\alpha \beta] = [\alpha' \beta']$. Since expansions are compositions of finitely many simple expansions, we can further assume that $\alpha',\beta'$ are simple expansions of $\alpha$ and $\beta$.

Now given two \s{G}-matrices $a,b$, we can assume  \s{\alpha= \begin{scriptsize}\begin{pmatrix}  
u_1 & u_2 & \cdots &u_k& \cdots & u_n \\
g_1 & g_2 & \cdots &g_k& \cdots  &g_n \\
v_1 & v_2 & \cdots &v_k& \cdots & v_n 
\end{pmatrix}  \end{scriptsize}} and \s{\beta= \begin{scriptsize}\begin{pmatrix}  
v_1 & v_2 & \cdots &v_k& \cdots & v_n \\
h_1 & h_2 & \cdots &h_k& \cdots  &h_n \\
w_1&w_2 & \cdots &w_k& \cdots  &w_n
\end{pmatrix}  \end{scriptsize}}.  Let 
$$\alpha'=\begin{scriptsize}\begin{pmatrix}  
u_1 & u_2 & \cdots &u_k0&u_k1& \cdots & u_n \\
g_1 & g_2 & \cdots  &g_{k, 0}&g_{k, 1}& \cdots &g_n \\
v_1&v_2 & \cdots &v_k(0)\sigma_{g_k} &v_k(1)\sigma_{g_k} & \cdots &v_n
\end{pmatrix}\end{scriptsize},~
 \beta'=\begin{scriptsize}
    \begin{pmatrix}
        v_1&v_2 & \cdots &v_k0&v_k1& \cdots &v_n\\
        h_1 & h_2 & \cdots  &  h_{k, 0}   &  h_{k, 1}    & \cdots &h_n\\
        w_1&w_2 & \cdots  &w_k(0)\sigma_{h_k} & w_k(1)\sigma_{h_k}    & \cdots  &w_n
    \end{pmatrix}
\end{scriptsize}$$
be the simple expansions of \s{\alpha} and \s{\beta} at \s{u_k} and \s{v_k} where where $((g_{k, 0}, g_{k, 1}),\sigma_{g_k}) = \phi(g_k)$ and $((h_{k, 0}, h_{k, 1}),\sigma_{h_k}) = \phi(h_k)$.
Then we have $$[\alpha'][\beta']=\begin{scriptsize}
    \begin{bmatrix}
        u_1 & u_2 & \cdots &u_k0&u_k1& \cdots & u_n \\
        g_1h_1 & g_2h_2 & \cdots  &g_{k, 0}h_{k, (0)\sigma_{g_k}}&g_{k, 1} h_{k, (1)\sigma_{g_k}}& \cdots &g_n \\
        w_1&w_2 & \cdots  &w_kx_0 & w_kx_1    & \cdots  &w_n
    \end{bmatrix}
\end{scriptsize} $$
where \s{x_i=(i)\sigma_{g_k}\sigma_{h_k},i\in\{0,1\}}.
Since 
\s{\phi(g_kh_k)=((g_{k, 0}h_{k, (0)\sigma_{g_k}},g_{k, 1} h_{k, (1)\sigma_{g_k}}),\sigma_{g_k}\sigma_{h_k})},
 we get \s{\alpha'\beta'} is an expansion of \s{\alpha\beta}, hence
\s{[\alpha'][\beta']=[\alpha][\beta]}.

 The associative law is obvious,  and for any \s{a= \begin{scriptsize}\begin{bmatrix}  
u_1 & u_2 & \cdots & u_n \\
g_1 & g_2 & \cdots  &g_n \\
v_1&v_2 & \cdots  &v_n
\end{bmatrix}  \end{scriptsize}} we have \s{ \begin{scriptsize}\begin{bmatrix}  
v_1 & v_2 & \cdots & v_n \\
g^{-1}_1 & g^{-1}_2 & \cdots  &g^{-1}_n \\
u_1&u_2 & \cdots  &u_n
\end{bmatrix}  \end{scriptsize}} as its inverse. This completes the proof that $V_\phi(G)$ is a group.
\end{proof}

\begin{definition}
    The group \s{V_\phi(G)} is called   \emph{the \s{\phi}-labeled Thompson group}. When \s{G} is trivial, \s{V_\phi(G)} is just the Thompson group \s{V}.
\end{definition}

\begin{remark}\label{vg right action}
    Note that there is an induced action of $V_\phi(G)$  on the Cantor set $\{0,1\}^\omega$ via the almost automorphisms on \s{\mathcal{T}}:
given \s{\alpha= \begin{scriptsize}\begin{bmatrix} 
u_1 & u_2 & \cdots & u_n \\
g_1 & g_2 & \cdots  &g_n \\
v_1&v_2 & \cdots  &v_n
\end{bmatrix}  \end{scriptsize}}, define \s{(u_iw) \alpha=v_i(w)\phi^*(g_i)} where $\phi^\ast$ is defined at the end of Subsection \ref{subsection:wrinj}. This defines a right \s{V_\phi(G)} action on \s{\mc}, we write \s{(w)\alpha} as the image of \s{w} under the action of \s{\alpha}.

\end{remark}

 \begin{remark}\label{rem:act-closy}
          When the induced map $\phi^\ast$ in the diagram (\ref{diagram1}) is injective, $V_\phi(G)$ is known as the R\"over--Nekrashevych group \cite{Nek04}. In \cite{SZ21} Skipper and Zaremsky  prove that all R\"over--Nekrashevych groups arise naturally from a cloning system  \cite{WZ18}.   In our setup, we have the flexibility that the action is not necessarily faithful. But one could also use the cloning system  to define our group \s{V_\phi(G)}. In fact \s{V_\phi(G)} can be defined by the system of  groups \s{G_n=G \wr \mathrm{S}_n} with suitable  cloning maps induced by \s{\phi}. In particular this gives another proof that $V_\phi(G)$ is a group.  We leave the details to the reader.

\end{remark}

\subsection{Tree diagrams and G-matrices} \label{subsection:tree-diagram-G-matx}
We discuss in this subsection how one can represent elements in   \s{\phi}-labeled Thompson groups by  $G$-labeled paired tree diagrams.

Recall that a \emph{finite  rooted binary tree} is a connected subtree of \s{\mathcal{T}} containing the root, where every vertex has degree \s{1} or \s{3} except the root which has degree $2$ (or degree 0 if the root is also a leaf). A vertex with degree $1$ is called a \emph{leaf}. A non-leaf vertex
together with the two vertices directly below it is called a \emph{caret}.

\begin{observation}
   The collection of partition sets of $\{0,1\}^\omega$ corresponds bijectively to the collection of  finite  rooted  binary trees, the operation of simple expansions on partition sets corresponds to adding carets to the corresponding trees. See  \cref{fig:expansion} for an example.
   \begin{figure}[htbp]
    \centering
\def\svgwidth{0.5\columnwidth} 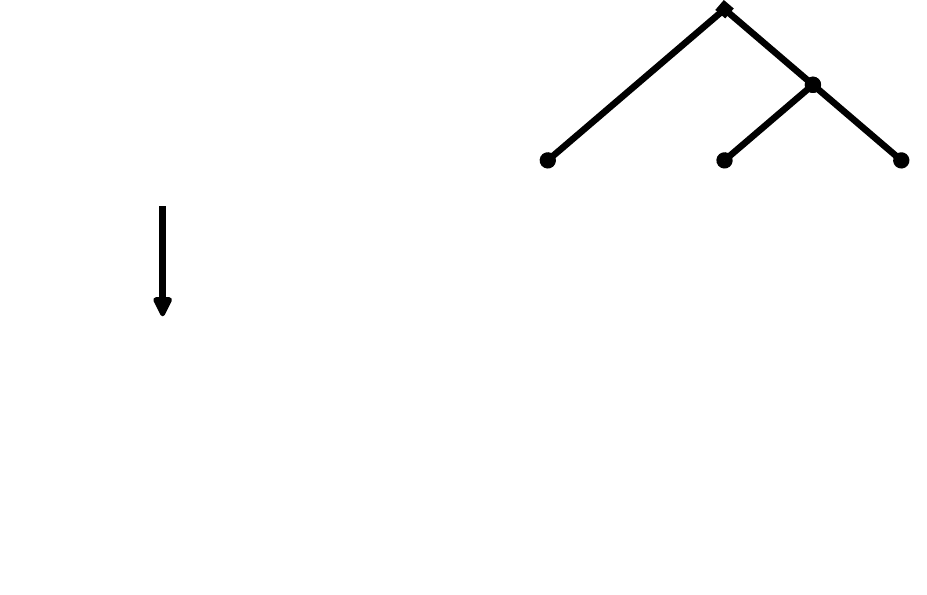
\caption{An example of the correspondence between partition set and trees under an expansion.}
\label{fig:expansion}
\end{figure}
\end{observation}

To see this, for every partition set \s{\Sigma},  there is a unique finite rooted   binary tree, denoted by \s{T_{\Sigma}},  such that \s{\Sigma} is precisely the leaves ( i.e. vertices of degree one) in \s{T_{\Sigma}}, or $T$ is the root when $\Sigma$ only contains the empty word.  And for each finite rooted binary  tree \s{T}, the set of leaves in \s{T} is exactly a partition set, denoted by \s{\Sigma(T)}. 
 
 Given two finite rooted binary  trees
\s{T_1, T_2}, we say \s{T_1}
is a \emph{simple expansion} of \s{T_2} if \s{T_1} is obtained from the \s{T_2} by adding a caret to one of the leaves;
we say \s{T_1} is an \emph{expansion} of \s{T_2}
if \s{T_1} is obtained from \s{T_2} by applying finitely many  simple expansions. Obviously \s{\Sigma_1} is a (simple) expansion of \s{\Sigma_2} if and only if \s{T_{\Sigma_1}} is a (simple) expansion of \s{T_{\Sigma_2}}.

    We can now represent elements of  \s{V_\phi(G)} by $G$-labeled paired tree diagrams. In fact, given a \s{G}-matrix \s{ a =\begin{scriptsize}\begin{bmatrix} 
u_1 & u_2 & \cdots & u_n \\
g_1 & g_2 & \cdots  &g_n \\
v_1&v_2 & \cdots  &v_n
\end{bmatrix}  \end{scriptsize}}, there are two trees $T_u$ and $T_v$ representing the partition sets $\{u_1,\cdots,u_n\}$ and $\{v_1,\cdots,v_n\}$. Let $\sigma$ be the corresponding bijection of leaves that maps $u_i$ to $v_i$ for $1\le i\le n$, then we can write $a$ as $(T_u, ((g_1,\cdots, g_n),\sigma)), T_v)$. In \cref{fig:correspondence}, we draw a $G$-labeled paired tree diagram with $T_u$ on the top and $T_v$ on the bottom where $\sigma$ is indicated by the arrow and $g_i$ is indicated by the labeling of the arrows.
\begin{figure}[h]
    \centering
\def\svgwidth{0.6\columnwidth} %% Creator: Inkscape 1.3.2 (091e20e, 2023-11-25, custom), www.inkscape.org
%% Accompanies image file 'correspondence.pdf' (pdf, eps, ps)
\begingroup%
  \makeatletter%
  \providecommand\color[2][]{%
    \errmessage{(Inkscape) Color is used for the text in Inkscape, but the package 'color.sty' is not loaded}%
    \renewcommand\color[2][]{}%
  }%
  \providecommand\transparent[1]{%
    \errmessage{(Inkscape) Transparency is used (non-zero) for the text in Inkscape, but the package 'transparent.sty' is not loaded}%
    \renewcommand\transparent[1]{}%
  }%
  \providecommand\rotatebox[2]{#2}%
  \newcommand*\fsize{\dimexpr\f@size pt\relax}%
  \newcommand*\lineheight[1]{\fontsize{\fsize}{#1\fsize}\selectfont}%
  \ifx\svgwidth\undefined%
    \setlength{\unitlength}{496.49441528bp}%
    \ifx\svgscale\undefined%
      \relax%
    \else%
      \setlength{\unitlength}{\unitlength * \real{\svgscale}}%
    \fi%
  \else%
    \setlength{\unitlength}{\svgwidth}%
  \fi%
  \global\let\svgwidth\undefined%
  \global\let\svgscale\undefined%
  \makeatother%
  \begin{picture}(1,0.41970136)%
    \lineheight{1}%
    \setlength\tabcolsep{0pt}%
    \put(0,0){\includegraphics[width=\unitlength,page=1]{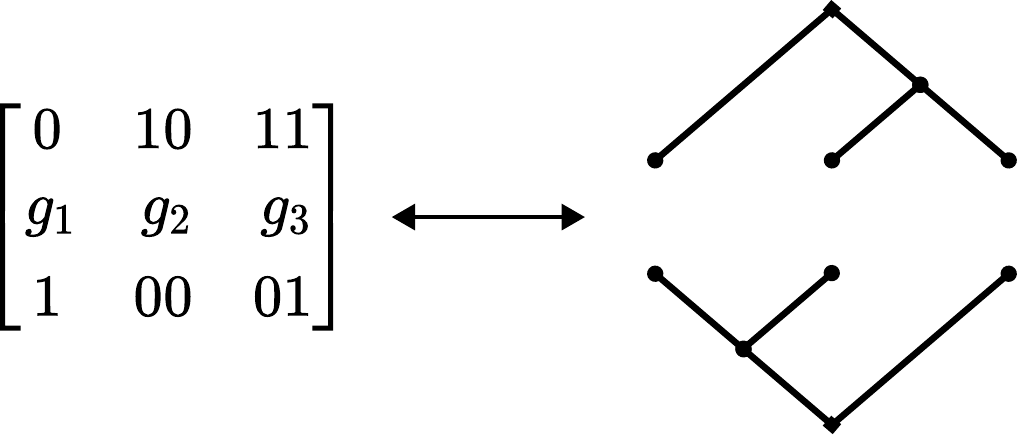}}%
    \put(0.61464341,0.29317755){\color[rgb]{0,0,0}\makebox(0,0)[lt]{\lineheight{1.25}\smash{\begin{tabular}[t]{l}$0$\end{tabular}}}}%
    \put(0.77772987,0.29466707){\color[rgb]{0,0,0}\makebox(0,0)[lt]{\lineheight{1.25}\smash{\begin{tabular}[t]{l}$10$\end{tabular}}}}%
    \put(0.9528403,0.29467888){\color[rgb]{0,0,0}\makebox(0,0)[lt]{\lineheight{1.25}\smash{\begin{tabular}[t]{l}$11$\end{tabular}}}}%
    \put(0.60863138,0.10546236){\color[rgb]{0,0,0}\makebox(0,0)[lt]{\lineheight{1.25}\smash{\begin{tabular}[t]{l}$00$\end{tabular}}}}%
    \put(0.78019574,0.10335343){\color[rgb]{0,0,0}\makebox(0,0)[lt]{\lineheight{1.25}\smash{\begin{tabular}[t]{l}$01$\end{tabular}}}}%
    \put(0.95885224,0.10546236){\color[rgb]{0,0,0}\makebox(0,0)[lt]{\lineheight{1.25}\smash{\begin{tabular}[t]{l}$1$\end{tabular}}}}%
    \put(0,0){\includegraphics[width=\unitlength,page=2]{correspondence.pdf}}%
  \end{picture}%
\endgroup%

\caption{An example of correspondence between $G$-matrices and tree diagrams.}
\label{fig:correspondence}
\end{figure}

Similarly, one can define simple expansion and expansion for a $G$-labeled paired tree diagrams using the  wreath recursion $\phi$. And we shall denote the equivalence class of such diagrams under expansion by $[T_u, ((g_1,\cdots, g_n),\sigma)), T_v]$. See Subsection \ref{subsection:lT-ptd} for more details.

\subsection{Reduction to injective wreath recursion}\label{subsection:red-inj}

In this subsection, we show that given a wreath recursion $\phi$, we may associate it with an injective wreath recursion and prove that the corresponding labeled Thompson groups are isomorphic. This allows us to reduce the study of the labeled Thompson group from a general wreath recursion to an injective one. Let us start with an motivating example.

\begin{example}\label{ex:ninj-wr}
    Let $G$ be a group, consider the wreath recursion given by $\phi(g) = (1,1,\mathrm{id})$. Then one checks straightforward that $V_{\phi}(G) \cong V$. This implies, for example, when $\phi$ is not injective, the finiteness properties of $V_{\phi}(G)$ may not be directly related to $G$.
\end{example}

Now let $G$ be a group and $\phi: G \to G\wr \mathrm{S}_2$ be any wreath recursion. If $\phi$ is not injective, let $G_1  = G/\mathrm{Ker}(\phi)$ and $\pi_1: G \to G_1$ be the corresponding quotient map. Then we can define a new map ${\pi}^w_1: G\wr \mathrm{S}_2\to G_1\wr \mathrm{S}_2$ by sending $((g,h),\sigma)$ to $((\pi_1(g),\pi_1(h)),\sigma)$. This further induces a new wreath recursion $ \phi_1 : G_1 \to G_1\wr \mathrm{S}_2$ which makes the following diagram commute:

\[\begin{tikzcd}
G \arrow{r}{\phi} \arrow[swap]{d}{\pi_1} & G\wr \mathrm{S}_2  \arrow{d}{{\pi}^w_1} \\
G_1 \arrow{r}{\phi_1} & G_1\wr \mathrm{S}_2
\end{tikzcd}
\]

%This diagram induces a natural homomorphism between the corresponding labeled Thompson groups $V_\phi \to V_{\phi_1}$ which turns out to be an isomorphism.It is clear that $\pi_1^V$ is surjective. To show it is an isomorphism, notice first that if an element $[T_1,\alpha, T_2]$ represents the identity element, then up to finitely many expansions, we can assume it is of the form $[T, ((1,\cdots, 1),\mathrm{id}),T]$

Of course, $\phi_1$ may again be not injective, so we have to do the above process again. Inductively, let $G_{i+1} = G_{i} / \mathrm{Ker}(\phi_{i})$ and $\pi_{i+1}: G_i \to G_{i+1}$ be the quotient map, and  $ \phi_{i+1} : G_{i+1} \to G_{i+1}\wr \mathrm{S}_2$ be the induced map. Let further  $\hat{G}  = \lim G_i$ and $\hat{\phi} :\hat{G} \to \hat{G} \wr \mathrm{S}_2$ be the induced map. We have the following commutative diagram:

\[\begin{tikzcd}
G \arrow{r}{\phi} \arrow[swap]{d}{\pi_1} & G\wr \mathrm{S}_2  \arrow{d}{{\pi}^w_1} \\
G_1 \arrow{r}{\phi_1} \arrow[swap]{d}{\pi_2} & G_1\wr \mathrm{S}_2 \arrow{d}{{\pi}^w_2}\\
G_2 \arrow{r}{\phi_2} \arrow[swap]{d}{\pi_3} & G_2\wr \mathrm{S}_2 \arrow{d}{{\pi}^w_3}\\
G_3  \arrow{r}{\phi_3}  & G_3\wr \mathrm{S}_2 \\
\vdots &  \vdots\\
\hat{G} \arrow{r}{\hat{\phi}} & \hat{G} \wr \mathrm{S}_2
\end{tikzcd}
\]

Let $\hat{\pi}: G\to \hat{G}$ the induced quotient map, $\pi_{i,j} = \pi_j \circ \cdots \pi_{i+1}$, and $\pi^w_{i,j} = \pi^w_j \circ \cdots \pi^w_{i+1}$. We also have a map $\hat{\pi}^w: G\wr \mathrm{S}_2 \to \hat{G} \wr \mathrm{S}_2$.

\begin{lemma}
   $\hat{\phi} :\hat{G} \to \hat{G} \wr \mathrm{S}_2$ is injective. 
\end{lemma}
\proof 
Suppose it is not injective, so we have some $\hat{g} \in \hat{G}$ such that $\hat\phi(\hat g) =1$. Firstly we have a $g \in G$ such that $\hat{\pi}(g) = \hat g$. By diagram commuting, we have $ \hat{\pi}^w(\phi(g)) =1$. Since $\hat{G} \wr \mathrm{S}_2 =  \lim G_i\wr \mathrm{S}_2$ under the structure map $\pi_i^w$, $\phi(g)$ must die in finitely many steps, i.e. $\pi^w_{0,k} (\phi(g)) =1$ for some $k\geq 1$. But this means $\pi_{0,k}(g)$ lies in the kernel of $\phi_k$, hence $\pi_{0,{k+1}}(g) =1$ and $\hat{g} =1$.
\qed

\begin{lemma}\label{lem:id}
    Let $g\in G$ such that $\hat\pi(g)=1$, then for any finite rooted binary tree $T$, $[T,((g,1,\cdots,1),\mathrm{id}),T]$ is the identity element in $V_\phi(G)$.
\end{lemma}

\proof 
Without loss of generality, we can assume that $T$ is the finite rooted binary tree $\bullet$ with a single root. Since $\hat\pi(g)=1$, we can find $k$ such that $\pi_{0,k}(g) =1$. If $k=1$, then by performing an expansion, we see that $[\bullet,((g),\mathrm{id}), \bullet]$ is equivalent to $[C,((1,1),\mathrm{id}), C]$ where $C$ is the caret, hence it is trivial. This further implies elements of the form $(C,((g_1,g_2),\mathrm{id}), C)$ must also be trivial for any $g_1,g_2 \in \mathrm{Ker}{\phi}$. Let us denote the normal subgroup $\mathrm{Ker}{\phi} \times \mathrm{Ker}{\phi} \times \mathrm{id}$ of $G\wr \mathrm{S}_2$ by $K_1$. Then for any element $g\in G$, if $\phi(g) \in K_1$, it must also be trivial. This is the same as saying that if $\pi_{0,2} (g)=1$, then  $[\bullet,((g),\mathrm{id}), \bullet]$  represents the trivial element. Inductively, we have if $\pi_{0,k} (g)=1$, then  $[\bullet,((g),\mathrm{id}), \bullet]$  represents the trivial element.
\qed

\begin{example}
    Let $G $ be the Baumslag--Solitar group $\langle a,t\mid ta^2t^{-1} = a^3\rangle$. Since $G$ is not hopfian, we have a surjection $f: G\to G$ with nontrivial kernel.  Define the wreath recursion $\phi: G\to G \wr S_2$ by mapping $g$ to $(f(g),f(g),\mathrm{id})$. Then $\mathrm{Ker}~\phi  = \mathrm{Ker}~f$ and one checks that the map $\phi_i$ will not stabilize in finitely many steps. 
\end{example}

So far, we have a commutative diagram of wreath recursion:

\[\begin{tikzcd}
G \arrow{r}{\phi} \arrow[swap]{d}{\hat\pi} & G\wr \mathrm{S}_2  \arrow{d}{\hat{\pi}^w} \\
\hat{G} \arrow{r}{\hat \phi} & \hat{G} \wr \mathrm{S}_2,
\end{tikzcd}
\]
where $\hat\pi^w : G\wr \mathrm{S}_2\to  \hat{G} \wr \mathrm{S}_2$ sending  $((g,h),\sigma)$ to $((\hat\pi(g),\hat\pi(h)),\sigma)$.   The diagram induces a map $\pi_V: V_\phi(G) \to V_{\hat\phi}(\hat{G})$ between the corresponding labeled Thompson groups. 

\begin{thm} \label{thm:inj-iso}
  Given any wreath recursion $\phi:G\to G \wr\mathrm{S}_2$, there is an induced surjection $\hat{\pi}: G\to \hat{G}$ and a corresponding injective wreath recursion $\hat{\phi}: \hat{G} \to \hat{G} \wr \mathrm{S}_2$ which together induces an isomorphism $\pi_V: V_\phi(G) \to V_{\hat\phi}(\hat{G})$ between the corresponding labeled Thompson groups. 
\end{thm}
\proof  It is left to show that $\pi_V$ is an isomorphism. It is clear that $\pi_V$ is surjective. We proceed to show that  $\pi_V$ is injective. One checks first by definition that for any injective wreath recursion $\hat\phi$, the identity element in $V_{\hat\phi}(\hat{G})$ must be of the form $[T,((1,\cdots,1),\mathrm{id}), T]$ where $T$ is a finite rooted binary tree. Now suppose that for some element $[T_1,((g_1,\cdots,g_n),\sigma), T_2]$, its image under $\pi_V$ is the identity, i.e. $[T,((1,\cdots,1),\mathrm{id}), T]$. On the other hand, by definition, $$\pi_V ([T_1,((g_1,\cdots,g_n),\sigma), T_2])  = [T_1, ((\hat\pi(g_1),\cdots \hat\pi(g_n)),\sigma), T_2].$$
This implies that $T_1 =T_2$, $\sigma =  \mathrm{id}$ and $\hat\pi(g_j) = 1$ for any $j$. Since 
$$[T_1,((g_1,\cdots,g_n),\mathrm{id}), T_1])  = \prod_{j=1}^n [T_1,((1,\cdots,g_j,\cdots,1),\mathrm{id}), T_1],$$
it suffices now to prove that $[T_1,((1,\cdots,g_j,\cdots,1),\mathrm{id}), T_1]$ is trivial for any $j$.
Note now that \\$[T_1,((1,\cdots,g_j,\cdots,1),\mathrm{id}), T_1]$ is conjugate to $[T_1,((g_j,1,\cdots,1),\mathrm{id}), T_1]$, we see that it is trivial by Lemma \ref{lem:id} and this finishes the proof.
\qed

\subsection{The  labeled Thompson group $V(G)$} \label{sec:V(G)}

Given a group $G$, recall that the diagonal wreath recursion   \s{\Delta: G\ra G \wr \mathrm{S}_2},  sends \s{g\mapsto ((g, g),id_{\mathrm{S}_2})}. In this case, we shall simply write \s{V_\Delta(G)} as \emph{\textit{\s{V(G)}}}, and call it the \emph{labeled Thompson group}.

Given any two groups $G,H$,  a group homomorphism $f: G\to H$ naturally induces a homomorphism  \s{V(f): V(G)\ra V(H)}, defined by sending \s{  \begin{scriptsize}\begin{bmatrix}  
u_1 & u_2 & \cdots & u_n \\
g_1 & g_2 & \cdots  &g_n \\
v_1&v_2 & \cdots  &v_n
\end{bmatrix}  \end{scriptsize}  } to \s{ \begin{scriptsize}\begin{bmatrix}  
u_1 & u_2 & \cdots & u_n \\
f(g_1) & f(g_2) & \cdots  &f(g_n) \\
v_1&v_2 & \cdots  &v_n
\end{bmatrix}  \end{scriptsize}}. 
Note that \s{V(-)} defines a functor  from the category of groups to itself.

\begin{lemma}
  The functor \s{V(-)}  maps surjections to surjections, injections to injections. 
\end{lemma}
\begin{proof}
    The surjective part is clear from definition. Let \s{f: G\ra H} be an injective group homomorphism, and suppose that \s{x= \begin{scriptsize}\begin{bmatrix}  
u_1 & u_2 & \cdots & u_n \\
g_1 & g_2 & \cdots &g_n \\
v_1& v_2 & \cdots &v_n
\end{bmatrix}  \end{scriptsize}\in ker(V(f))}, then we have

$$\begin{scriptsize}\begin{bmatrix}  
u_1 & u_2 & \cdots & u_n \\
f(g_1) & f(g_2) & \cdots &f(g_n) \\
v_1& v_2 & \cdots &v_n
\end{bmatrix}  = 1 \end{scriptsize} $$

Since the identity element in $V(H)$ must be of the form   
\s{ \begin{scriptsize}\begin{bmatrix}  
w_1 & w_2 & \cdots & w_m \\
1_H & 1_H & \cdots &1_H \\
w_1& w_2 & \cdots &w_m
\end{bmatrix}  \end{scriptsize} }
(up to expansion) for some partition set \s{\{w_1,  \cdots , w_m\}}, we conclude first that \s{f(g_i)=1_H} (here we rely on the fact that the wreath recursion $\Delta$ is injective), then \s{u_i=v_i} for all \s{1\leq i\leq n}.  But $f$ is injective, so $g_i=1_G$ for all $i$. Hence \s{x=1_{V(G)}}. 
\end{proof}

Another way to see \s{V(G)} is to view it as a subgroup of a wreath product which we explain now.
Recall that the Thompson group \s{V} acts on \s{\mc}, 
let \s{G^{\mc}\rtimes V} be the unrestricted wreath product of \s{G} by \s{V}. Then we can define a group embedding \s{i_G: V(G)\ra G^{\mc}\rtimes V } in the following way:   $$i_G:\begin{scriptsize}\begin{bmatrix}  
u_1 & u_2 & \cdots & u_n \\
g_1 & g_2 & \cdots  &g_n \\
v_1&v_2 & \cdots  &v_n
\end{bmatrix}  \end{scriptsize}\mapsto
( f,\begin{scriptsize}\begin{bmatrix}  
u_1 & u_2 & \cdots & u_n \\
v_1&v_2 & \cdots  &v_n
\end{bmatrix}  \end{scriptsize})$$
 where \s{f\in G^\mc} is defined by \s{f(x)=g_i} for $x\in u_i\{0,1\}^{\omega}$. In fact, to verify it is indeed a group homomorphism we can take a direct computation: $$
 \begin{aligned}
     i_G\bigg(\begin{bmatrix}  
u_1 & u_2 & \cdots & u_n \\
g_1 & g_2 & \cdots  &g_n \\
v_1&v_2 & \cdots  &v_n
\end{bmatrix}\bigg)i_G\bigg(   \begin{bmatrix}  
v_1 & v_2 & \cdots & v_n \\
h_1 &h_2 & \cdots  &h_n \\
w_1&w_2 & \cdots  &w_n
\end{bmatrix}\bigg)&=(f_1, \sigma_1 )(f_2, \sigma_2  )\\
&=(f_1 \cdot (f_2) \sigma_1, \sigma_1\sigma_2)\\
&=i_G\bigg(\begin{bmatrix}  
u_1 & u_2 & \cdots & u_n \\
g_1h_1 & g_2h_2 & \cdots &g_nh_n \\
w_1&w_2 & \cdots   &w_n
\end{bmatrix}\bigg)
 \end{aligned}
 $$
 where \s{f_1(x)=g_i} for \s{x\in u_i\xinf} and \s{f_2(x)=h_i} for \s{x\in v_i\xinf}; And \s{\sigma_1=\begin{bmatrix}  
u_1 & u_2 & \cdots & u_n \\
v_1&v_2 & \cdots   &v_n
\end{bmatrix}} and \s{\sigma_2=\begin{bmatrix}  
v_1 & v_2 & \cdots & v_n \\
w_1&w_2 & \cdots   &w_n
\end{bmatrix}}. The last equality follows from 
 $$(x) (f_1 \cdot (f_2) \sigma_1 )= (x) f_1 \cdot (x)((f_2) \sigma_1 )= (x) f_1 \cdot ((x)\sigma_1)f_2 = g_ih_i, \forall x\in u_i\xinf.$$ 
Of course, here we assumed $\sigma_1$ and $\sigma_2$ are already compossible but this is not a problem after some simple expansions. 

Let $p_G:G\to \{1\}$ be the group homomorphism that maps all elements in $G$ to $1$. We then have an induced map $V(p_G):V(G) \to V(\{1\}) = V$. Denote the kernel by $lim_{\mathcal{T}}G$, then  we have a short exact sequence 
$$
1\ra lim_{\mathcal{T}}G \ra V(G) \xrightarrow{V(p_G)} V \ra 1
$$

Note that one can also describe the kernel $\lim_{\mathcal{T}} G$ as the following: $G_n = \oplus_{2^n} G$ with injection maps $G_n \to G_{n+1}$ given by $(g_1,g_2\cdots g_{2^n}) \to (g_1,g_1,g_2,g_2,\cdots, g_{2^n}, g_{2^n})$, then $\lim_{\mathcal{T}} G= \lim_{n\to \infty} G_n$. View $G^{\mc}$ as the group of functions from $\mc$ to $G$, then the image of $lim_{\mathcal{T}}G$ under the isomorphism $i_G$  is precisely the subgroup of all continuous functions from $\mc$ to $G$, where $G$ is given the discrete topology. Furthermore, we have a splitting for $V(p_G)$, \s{V(r_G): V=V(\{1\})\ra V(G)} induced by the unique injection $r_G:\{1\} \to G$.  This way we also get the splitting \s{V(G)= lim_{\mathcal{T}}G\rtimes V}.

\subsection{The quasi-retract from \s{V(G)} to \s{G}} We define a quasi-retract from \s{V(G)} to \s{G} which implies $G$ is of type $F_n$ if $V(G)$ is. Recall that a function $f:X\to Y$ between two metric spaces is called \emph{coarse Lipschitz} if there exist constants $C,D>0$ so that 
$$d(f(x),f(x')) \leq Cd(x,x')+D \text{ for all } x,x'\in X$$
For example, any homomorphism between finitely generated groups is coarse Lipschitz with respect to the word metrics.
A function $\rho:X\to Y$ is said to be a \emph{quasi-retraction} if it is coarse Lipschitz and there exists a coarse Lipschitz function $\iota:Y\to X$ and a constant $E>0$ so that $d(\rho\circ \iota(y),y)\leq E$ for all $y\in Y$. If such a function $\rho$ exists, $Y$ is said to be a \emph{quasi-retract} of $X$. In this case, it is not hard to show that $\iota$ must be a quasi-isometric embedding of $Y$ into $X$.

\begin{thm}\cite[Theorem 8]{Al94}\label{quasi-re}
    Let $G$ and $Q$ be finitely generated groups such that $Q$ is a quasi-retract of $G$ with respect to word metrics corresponding to some finite generating sets. Then if $G$ is of type $F_n$, so is $Q$. 
\end{thm}

To apply \Cref{quasi-re}, we first define a set-theoretic embedding $\iota_G: G \to V(G)$ by mapping $g$ to $[C,((g,1),\mathrm{id}),C]$ where $C$ is the caret. We also have a map $\rho_G: V(G)\to G$ by mapping an element $[T,((g_1,\cdots, g_n),\sigma),S]$ in $V(G)$ to $g_1$. One checks that this is a well-defined map, but not a group homomorphism in general. We do have $\rho_G \circ \iota_G =\mathrm{id}_G$. Recall that $V$ embeds naturally in $V(G)$ as the subgroup of $V(G)$ with all elements have labeling $1$.

\begin{lemma}\label{lem:V(G)-gen-set}
     \s{V(G)} is generated by $V$ and $\iota_G(G)$.
\end{lemma}
\begin{proof}
Let $H$ be the subgroup generated by $V$ and $\iota_G(G)$. Given any element\\ $[T,((g_1,\cdots,g_n),\sigma),S]$, we can write it as the product 
$$[T,((g_1,\cdots,g_n),\mathrm{id}),T][T,((1,\cdots,1),\sigma),S].$$
The right hand already lies in $V$, so it suffices to prove that elements of the form
$[T,((g_1,\cdots,g_n),\mathrm{id}),T]$ lies in $H$. Let $T_r$ be a finite rooted binary tree obtained by adding a finite rooted binary tree at the right leaf of $C$, then elements of the form $[T_r,((g_1,\cdots, g_n),\mathrm{id}), T_r]$ where $g_i =1$ for all but one $i$, are conjugate to elements of the form $[T_r,((g,1,\cdots, 1),\mathrm{id}), T_r]$ by elements in $V$. But $[T_r,((g,1,\cdots, 1),\mathrm{id})), T_r] =[C,((g,1),\mathrm{id}),C]$, so elements of the form  $[T_r,((g_1,\cdots, g_n),\mathrm{id})), T_r]$ where $g_i =1$ for all but one $i$ belong to $H$. Therefore elements of the form $[T_r,((g_1,\cdots, g_n),\mathrm{id})), T_r]$ also lie in $H$ since $H$ is closed under multiplication. In particular $[C,((1,g),\mathrm{id}),C] \in H$ for any $g\in G$. Now the same argument before would imply all elements of the form $[T_l,((g_1,\cdots, g_n),\mathrm{id}), T_l]$ where $T_l$ is obtained from $C$ by adding a finite rooted binary tree at its left leave. Since any elements of the form $[T,((g_1,\cdots,g_n),\mathrm{id}),T]$ can be written as a product of the form $[T_r,((h_1,\cdots, h_{n_r}),\mathrm{id}), T_r][T_l,((t_1,\cdots, t_{n_l}),\mathrm{id}), T_l] $, they must also lie in $H$.
\end{proof}

\begin{thm}\label{quasi-retract}
   The group \s{G} is a quasi-retract of \s{V(G)} when $G$ is finitely generated. In particular $\iota:G\to V(G)$ is a quasi-isometric embedding when $G$ is finitely generated.
\end{thm}
\begin{proof}
Fix a generating set $S_G$ for $G$, $S_V$ for $V$. By \Cref{lem:V(G)-gen-set}, $S_V$ and $\iota_G(S_G)$ generates $V(G)$. We proceed to show that the map $\rho_G: V(G) \to G$ is coarse Lipschitz with respect to the word metric on $V(G)$ and $G$. Note that:
 \begin{enumerate}
    \item  $\rho_G( x ~\iota_G(s))\in \{\rho_G(x),  \rho_G(x)s\}$ for all $s\in S_G$ and $x\in V(G)$,  and 
    \item  $\rho_G(xy) = \rho_G(x)$ for any $y\in V$ and $x\in V(G)$. 
\end{enumerate}

It follows that $\rho_G$ is non-expanding and hence coarse Lipschitz.    
\end{proof}

\begin{corollary} \label{cor:fin-VG-imply-G}
    If $V(G)$ has type $F_n$, so does $G$.
\end{corollary}
\proof 
When $n\ge 2$, this follows from \Cref{quasi-re} and \Cref{quasi-retract}. For $n=1$, suppose $G$ is not finitely generated, then $G$ can be written as a proper infinite increasing union $\bigcup_{i=1}^\infty G_i$. So $V(G)$ can also be written as a proper infinite increasing union $\bigcup_{i=1}^\infty V(G_i)$ which implies $V(G)$ is not finitely generated.
\qed

\begin{remark}\label{rem:RN_gp}
    Our proof relies on the existence of the quasi-retraction map $\rho_G$ which does not exist for general injective wreath recursions $\phi$. In fact, if one takes $G$ to be the Grigorchuk group which is only finitely generated, the corresponding R\"over--Nekrashevych group is of type $F_\infty$ \cite{BM16}.  
\end{remark}

\subsection{Relation with Thompson's Splinter group}\label{sec:rel-spli} We show that the labeled Thompson group $V(G)$ is isomorphic to Thompson's Splinter group \cite{Tho80} in this subsection. Recall that we identified the Cantor set $\mathfrak{C}$ with $\{0,1\}^\omega$ and Thompson group $V$ is a subgroup of the homeomorphism group of $\mathfrak{C}$. Given a set $X$, let $\mathrm{Perm}(X)$ denote the permutation group of $X$.

\begin{definition}\cite[Definition 2.1]{Tho80}
     Let \s{G} be a subgroup of $\mathrm{Perm}(X)$. Then \emph{the splinter group} $\mathrm{Sp}_X(G, \mathfrak{C})$, is the subgroup of $\mathrm{Perm}{(X\times \mathfrak{C})}$ generated by the following two types of elements: 
\begin{enumerate}
    \item $\bar{g}$ where $g\in G$ and $\bar{g}$ acts on $(x,w)\in X\times \mathfrak{C}$ by  $ (x,w)\bar{g} = ((x)g,w)$ if $w\in 01\{0,1\}^\omega$, otherwise $(x,w)\bar{g} = (x,w)$. 
    \item $\bar{v}$ where $v$ is an element of the Thompson group $V$ and $\bar{v}$ acts on $(x,w)\in X \times \mathfrak{C}$ by $(x,w)\bar{v} =(x,(w)v)$. 
\end{enumerate}
\end{definition}

Note that in the above definition, only the group structure of $G$ plays a role, the role of $X$ is not so important. In fact, for any set $X$ which $G$ acts faithfully, the corresponding splinter group $\mathrm{Sp}_X(G, \mathfrak{C})$ is always isomorphic to the labeled Thompson group $V(G)$.

\begin{thm}\label{thm:split-vs-lab}
    For any set $X$ and group $G\leq \mathrm{Perm}(X)$,  $\mathrm{Sp}_X(G, \mathfrak{C})$ is isomorphic to $V(G)$.
\end{thm}
\begin{proof}
We define a group isomorphism $F$ from $V(G)$ to $\mathrm{Sp}_X(G, \mathfrak{C})$ as the following. For each element $a=[T,((g_1,\cdots,g_n),\sigma),S] \in V(G)$, we define the permutation $F_a$ on $X \times \mathfrak{C}$ as follows: $(x,u_i w)F_a \coloneqq (xg_i, (u_i)\sigma w)$ for $x \in X$, $u_i \in \Sigma(T)$ and $w \in \{0,1\}^\omega$. It is straightforward to verify that $F_a$ is independent of the choice of representative $(T,((g_1,\cdots,g_n),\sigma),S)$ by checking that $F_a$ is invariant under simple expansions. For $a=[T,((g_1,\cdots,g_n),\sigma),S]$, $b=[S,((h_1,\cdots,h_n),\tau),R]$ and $ab=[T,((g_1h_{(1)\sigma},\cdots,g_nh_{(n)\sigma}),\sigma\tau), R]$ in $V(G)$, we have that$$(x,u_i w)F_{ab} = (xg_ih_{(i)\sigma}, (u_i)\sigma\tau w)= (x,u_i w)F_{a}F_b $$ for $x \in X$, $u_i \in \Sigma(T)$. So the assignment $F$ sending $a$ to $F_a$ is a group homomorphism from $V(G)$ to $\mathrm{Perm}{(X\times \mathfrak{C})}$.

We now proceed to show that $F$ is injective and $F(V(G))=\mathrm{Sp}_X(G, \mathfrak{C})$. For every element $g\in G$, let $a_g=[T_0 ,((1_G,g,1_G);id), S_0]\in V(G)$ where $T_0=S_0=\{00\{0,1\}^{\omega}, 01\{0,1\}^{\omega}, 1\{0,1\}^{\omega}\}$  and let $V$ denote the subgroup of $V(G)$ with trivial labeling. Then $F_{a_g}=\bar{g}$ and $F_v=\bar{v}$, for $g\in G$ and $v\in V$. The exact same proof in \Cref{lem:V(G)-gen-set} implies that $V(G)$ is generated by $\{a_g \colon g \in G\}$ and $V$. Consequently, the image of $V(G)$ under $F$ indeed lies in the Splinter group $\mathrm{Sp}_X(G,\mathfrak{C})$ and the map $F \colon a \rightarrow F_a$ defines a surjective homomorphism from $V(G)$ to $\mathrm{Sp}_X(G, \mathfrak{C})$. For each $a=[T,((g_1,\cdots,g_n),\sigma),S] \in V(G)$ such that $F_a=id_{\mathrm{Sp}_X(G, \mathfrak{C})}$, we have $(x,u_i w)F_a = (xg_i, (u_i)\sigma w)=(x,u_i w)$ for $x\in X$, $u_i \in \Sigma(T)$, $i=1, \cdots, n$ and $w\in \{0,1\}^{\omega}$. Then $g_i=1_G$ for $i=1,\cdots,n$, since $(x)g_i =x $ holds for all $x \in X$ and $G$ acts faithfully on $X$. Also, $S=T$ and $\sigma=id$, since $(u_i)\sigma=u_i$ for $i=1, \cdots , n$. So $a=[T,(1,\cdots,1),id), T]=1_{V(G)}$ and $F$ is an injection, hence an isomorphism.
\end{proof} 

Recall that an embedding of a group $G$ into a group $H$ is a \emph{Frattini embedding} if whenever two elements of $g, g^{\prime} \in G$ are conjugate in $H$, they are already conjugate in $G$. By \cite[Theorem 2.6 and Proposition 2.7]{Tho80}, we have the following corollaries:

\begin{corollary}
    The embedding from $G$ to $V(G)$ is a Frattini embedding.
\end{corollary}

\begin{corollary}\label{cor:wps-vg}
     Let $G$ be a finitely generated group, then the word problem for $G$ is solvable if and only if it is solvable for $V(G)$.
\end{corollary}
\begin{remark}
    Thompson uses crucially the Splinter group, in particular, Corollary \ref{cor:wps-vg} in his proof of the Boone--Higman--Thompson  Theorem \cite{BH74, Tho80}:  a
finitely generated group has solvable word problem if and only if it embeds into a finitely generated simple subgroup of a finitely presented group. 
\end{remark}

\section{Normal subgroups of \s{V(G)} and its uniform perfection}

We study the normal subgroup structure and uniform perfection of $V(G)$ in this section. Note that the normal subgroup structure of $F(G)$ has already been studied in \cite[Section 4]{Tan16}. Recall that we  view $V$ as the subgroup of $V(G)$ consisting of all elements with trivial labeling.

    Given \s{g\in G,  u\in \{0, 1\}^{<\omega}}, we define an element of \s{\vdeltag} by \s{\Lambda_u(g)= \begin{scriptsize}\begin{bmatrix}  
u & u_2 & \cdots & u_n \\
g & 1 & \cdots &1 \\
u& u_2 & \cdots &u_n
\end{bmatrix}  \end{scriptsize}} where \s{\{u, u_2,  \cdots , u_n\}} is a partition set. The definition of \s{\Lambda_u(g)} does not depend on the choice of \s{u_2,  \cdots , u_n},  thus is well-defined. Recall that we have a projection map $V({p_G}): V(G)\to V$ induced by the map $p_G:G\to 1$. Let us abbreviate  $V(p_G)$ by $pr_G$. 

\begin{thm}
    Any proper normal subgroup of \s{V(G)} is a subgroup of \s{lim_{\mathcal{T}}G}. 
\end{thm}
\begin{proof}
   Let \s{N} be a normal subgroup of \s{V(G)},  then \s{pr_G(N)\trianglelefteq V}, hence \s{pr_G(N)=V} or \s{1}. 

Assume that \s{pr_G(N)=V}.   Then \s{N} has an element \s{g=\begin{bmatrix}
           u& \cdots \\
           h& \cdots \\
           v& \cdots 
       \end{bmatrix}} where \s{\{u, v\}} is an anti-chain, i.e., \s{u\xstar\cap v\xstar=\varnothing} and \s{u\xstar\cup v\xstar} is a proper subset in \s{\xstar}.
       For any anti-chain \s{\{r, s\}},  let \s{q=\begin{bmatrix}
           u&v& \cdots \\
           h&1_G& \cdots \\
           r&s& \cdots 
       \end{bmatrix}}, and \s{p=q^{-1}gq},  then  \s{p} has the form \s{p=\begin{bmatrix}
           r& \cdots \\
           1_G& \cdots \\
           s& \cdots 
       \end{bmatrix}}. Now for any \s{f\in G}, $$p^{-1}\Lambda_
r(f)p=\begin{bmatrix}
    s&s_2& \cdots &s_n\\
    1_G&g^{-1}_2& \cdots &g^{-1}_n\\
    r&r_2& \cdots &r_n
\end{bmatrix}
\begin{bmatrix}
  r&r_2& \cdots &r_n\\ 
  f&1_G& \cdots &1_G\\
  r&r_2& \cdots &r_n
\end{bmatrix}
\begin{bmatrix}
    r&r_2& \cdots &r_n\\
    1_G&g_2& \cdots &g_n\\
     s&s_2& \cdots &s_n
\end{bmatrix}
    =\Lambda_s(f)   $$ Hence \s{[p, \Lambda_r(f)]=\Lambda^{-1}_s(f)\Lambda_r(f)
\in N
    } for any anti-chain \s{\{r, s\}}. Let \s{\pi: G\ra G/N} be the quotient map,  then since \s{\{r, s0\}} \s{\{r, s1\}}, and \s{\{s0, s1\}} are anti-chains, we have
    $$\pi(\Lambda_r(f))=\pi(\Lambda_s(f))=\pi(\Lambda_{s0}(f))\pi(\Lambda_{s1}(f))=\pi(\Lambda_{r}(f))^2$$ Hence \s{\pi(\Lambda_r(f))=1} and \s{\Lambda_r(f)\in N}. 
Since \s{lim_{\mathcal{T}}G} is generated by elements of the form \s{\Lambda_r(f)}, we have \s{lim_{\mathcal{T}}G\leq N}. Hence \s{N=V(G)} since \s{pr_G(N)=V}. 

So if \s{N} is proper normal subgroup, \s{pr_G(N)= \{1\}} which means that \s{N\leq lim_{\mathcal{T}}G}. 
\end{proof}

\begin{remark}\label{rem:add-mach}
For a general injective wreath recursion $\phi$, the normal subgroup structure of $V_\phi(G)$ can be complicated.  For example, take $G= \langle t\rangle \cong \mathbb{Z}$ and consider the adding machine in \cite[1.7.1]{Nek05}. In our notation, the corresponding wreath recursion is \s{\phi(t)=((1,t),\sigma)}, where \s{\sigma} is the nontrivial permutation of \s{\{0,1\}}. Then by \cite[Corollary 5.4]{Nek18}, the abelization of the group \s{V_\phi(\mathbb{Z})}  is \s{\mz}.  
\end{remark}

 Since $V$ is an infinite group, we have the following corollary.

\begin{corollary}\label{cor-VG-perfect}
    \s{\vdeltag} has no proper finite index subgroup. In particular it is perfect. 
\end{corollary}

Given $g,h$ in a group, we shall use $g^h$ to denote the conjugation  $h^{-1}gh$. We further improve \Cref{cor-VG-perfect} to the following.

\begin{thm}\label{thm:unfm-perfect}
    $\vdeltag$ is $5$-uniformly perfect.
\end{thm}
\begin{proof}
Let us switch to $G$-labeled  paired tree diagrams to represent elements in $V(G)$.
    For any element $[T_-, ((g_1, \cdots, g_n), \sigma),  T_+] \in \vdeltag$, we have $[T_-, ((g_1, \cdots, g_n), \sigma) , T_+]$
    $$=[T_-, ((1_G,g_2, \cdots, g_n), \mathrm{id}),  T_-][T_-, ((g_1, 1_G\cdots, 1_G),\mathrm{id}),  T_-][T_-, ((1_G, \cdots, 1_G), \sigma) , T_+].$$
    
    Notice that $[T_-, ((1_G, \cdots, 1_G),\sigma) , T_+]$ is an element in Thompson group $V$ and $V$ is $3$-uniformly perfect \cite[Corollary 6.6]{GG17}. So $[T_-, ((1_G, \cdots, 1_G), \sigma) , T_+]$ is a product of at most $3$ commutators in $V(G)$ since $V$ is a subgroup of $V(G)$.
    
    Let $s_{12}=(12) \in \mathrm{S}_{n}$ and $\bar{s_{12}}=((1_G, \cdots, 1_G),s_{12}) \in G \wr \mathrm{S}_n $, then we have an element  $[T_-,\bar{s_{12}}, T_-] \in \vdeltag $. Note that
    $$[T_-, ((g_1, 1_G\cdots, 1_G), \mathrm{id}),  T_-]^{[T_-,\bar{s_{12}}, T_-]}=[T_-,((1_G, g_1,1_G,  \cdots,1_G), \mathrm{id} ), T_-].$$ Thus to prove $[T_-, ((g_1, \cdots, g_n), \sigma),  T_+]$ can be written as products of $5$ commutators, it suffices to prove elements of the form $[T, ((1_G,g_2, \cdots, g_n), \mathrm{id}),  T]$ are commutators. 

   Let us prove now $v=[T, ((1_G,g_2, \cdots, g_n), \mathrm{id}),  T]$ is a commutator. Let $T^\prime$ be the tree obtained from $T$ by adding $(n-1)$ new carets such that each new caret always attaches to the left most leaf, then $v=[T^\prime, ((1_G, \cdots ,1_G,g_2,\cdots, g_n),\mathrm{id}), T^\prime]$.  Let $T^{\prime\prime}$ be the tree obtained from $T$ by adding $(n-1)$ new carets on each leaf except the left most one, we have $v$ also equals to $[T^{\prime\prime}, ((1_G, g_2,g_2, \cdots, g_n,g_n),\mathrm{id}), T^{\prime\prime}]$. Now let $\alpha \in \mathrm{S}_{2n-1}$ with $\bar{\alpha}=((1_G, \cdots, 1_G), \alpha) \in G \wr \mathrm{S}_{2n-1}$ such that $\bar{\alpha} ((1_G, \cdots ,1_G,g_2,\cdots, g_n),\mathrm{id}) \bar{\alpha}^{-1}=((1_G,g_2,1_G,g_3,\cdots,1_G,g_n,1_G),\mathrm{id})$. Then we have 
   $$vav^{-1}a^{-1}=[T^{\prime\prime}, ((1_G,1_G, g_2, 1_G, g_3, \cdots, 1_G,g_n),\mathrm{id}), T^{\prime\prime}], $$ where $a=[T^{\prime\prime}, \bar{\alpha}, T^\prime] $.
    Similarly pick $\beta \in \mathrm{S}_{2n-1}$ with $\bar{\beta}=((1_G, \cdots, 1_G), \beta) \in G ~\wr ~\mathrm{S}_{2n-1}$ such that
    $$\bar{\beta} ((1_G,1_G, g_2, 1_G, g_3, \cdots, 1_G,g_n),\mathrm{id})\bar{\beta}^{-1}=((1_G, \cdots ,1_G,g_2,\cdots, g_n),\mathrm{id}).$$ 
     Then $b(vav^{-1}a^{-1})b^{-1}=[T^\prime,((1_G, \cdots ,1_G,g_2,\cdots, g_n),\mathrm{id}),T^\prime]=v$, where $b=[T^\prime, \bar{\beta}, T^{\prime\prime}]$. Thus we have $v$ is a commutator and we finish the proof.   
\end{proof}

\section{Finiteness properties of  labeled Thompson groups}\label{Finiteness of vg}
The goal of this section is to prove the following theorem:

\begin{thm}\label{finiteness-main}
    Let $\phi: G \to G\wr \mathrm{S}_2$ be any injective wreath recursion. If \s{G} is of type \s{F_m}, so is \s{V_\phi(G)}.
\end{thm}
\begin{remark}
\begin{enumerate}

  \item   If the wreath recursion $\phi$ is not injective, one could first use Theorem \ref{thm:inj-iso} to reduce it to an injective one, then apply Theorem \ref{finiteness-main} to get finiteness property of $V_\phi(G)$ from that of $\hat{G}$ in Theorem \ref{thm:inj-iso}. However, it should be mentioned that it can be very difficult to determine the finiteness property of $\hat{G}$. 
     \item We remind the reader that for a general injective wreath recursion, $V_\phi(G)$ might have better finiteness property than $G$, see Remark \ref{rem:RN_gp}.  
\end{enumerate}   
\end{remark}

The proof of Theorem \ref{finiteness-main} follows the standard approach to determine finiteness properties of Thompson-like groups using paired forest diagrams, see for example \cite{BFM+16,Fa03,SW23,SWZ19, SZ21}. The standing assumption of this section is that $\phi$ is an injective wreath recursion.

\subsection{Labeled Thompson groups and paired forest diagrams} \label{subsection:lT-ptd} Recall in Subsection \ref{subsection:tree-diagram-G-matx}, we have explained how to define a $\phi$-labeled Thompson group using $G$-labeled paired tree diagrams. We expand this now to $G$-labeled paired forest diagram.

Recall that a \emph{finite rooted binary tree} is
a finite tree such that every vertex has degree 3, except the \emph{leaves}, which have degree 1, and
the \emph{root}, which has degree 2 (or degree 0 if the root is also a leaf). We will drop the word finite whenever it is clear from the context. And we will draw such trees
with the root at the top and the nodes descending from it, down to the leaves. A non-leaf vertex
together with the two vertices directly below it is called a \emph{caret}. If the leaves of a caret in \s{T} are
leaves of \s{T} , we will call the caret\emph{ elementary}.
A \emph{forest} is a finite linearly ordered union of  rooted binary trees. If each tree in the forest is either a single root or a caret, we call the forest \emph{elementary}. 

\begin{definition}
     A \emph{$G$-labeled paired forest  diagram} is  a triple \s{(F_-,(\vec{g},\sigma),F_+)} consisting of forests \s{F_-, F_+} with the same number of leaves, say $n$,  and \s{(\vec{g},\sigma)\in G \wr \mathrm{S}_n}. When \s{F_-,F_+} both have only one component, it is the \emph{$G$-labeled paired tree diagram}.
\end{definition}

We can define the expansion and reduction rule for $G$-labeled paired forest  diagrams just as we did for $G$-labeled paired tree diagrams.  Given a $G$-labeled paired forest  diagram $\left(F_{-}, (\vec{g},\sigma), F_{+}\right)$, and some $1 \leq i \leq n$, suppose $\vec{g} = (g_1,\cdots, g_n)$ and $\phi(g_i) = ((g_{i0},g_{i1}), \sigma_{g_i})$. Then one first replaces $F_{ \pm}$ with forests $F_{ \pm}^{\prime}$ obtained from $F_{-}$ (resp. $F_+$) by adding a caret at the  $i$-th leaf (resp. $(i)\sigma$-th leaf) and  
 $((g_1,\cdots,g_i,\cdots,g_n),\sigma)$ with    $((g_1,\cdots,g_{i-1},g_{i0},g_{i1},\cdots,g_n) ,\sigma' )$. Here $\sigma'$ is obtained from $\sigma$ as follows: first we replace the set $\{1,\cdots n\}$ (which we identify with the set of leaves of $F_-$) by $\{1,\cdots, i0,i1,\cdots n\}$ where $i$ is replaced by $i0$ and $i1$. And similarly, replace $\{1,\cdots n\}=\{(1)\sigma,\cdots (n)\sigma\}$ by $\{(1)\sigma,\cdots,{(i)\sigma 0},{(i)\sigma 1},\cdots (n)\sigma\}$.  Define now $(j)\sigma' = (j)\sigma$ for $j\neq i$, and $(ik)\sigma' = (i)\sigma(k)\sigma_{g_i}$ for $k\in \{0,1\}$. Of course, to make $\sigma'\in \mathrm{S_{n+1}}$, one has to first order the leaves of $F_-'$ and $F_+'$ lexicographically (since each leave is a word in $\{0,1\}$, we can order the leaves of a finite rooted binary tree lexicographically assuming $0<1$), and identify them with $\{1,\cdots, n+1\}$ according to the order, the bijection between the leaves now gives the actual element $\sigma'\in \mathrm{S_{n+1}}$. Let $\vec{g}' = (g_1,\cdots,g_{i-1},g_{i0},g_{i1},\cdots,g_n)$, the triple $\left(F_{-}', (\vec{g}',\sigma'), F_{+}'\right)$ is called the \emph{expansion} of $\left(F_{-}, (\vec{g},\sigma), F_{+}\right)$ at the $i$-th leaf. Conversely,  $\left(F_{-}, (\vec{g},\sigma), F_{+}\right)$ is  a \emph{reduction} of $\left(F_{-}^{\prime}, (\vec{g}',\sigma'), F_{+}^{\prime}\right)$. Two paired forest diagrams are \emph{equivalent} if one is obtained from the other by
a sequence of reductions or expansions. The equivalence class of \s{(F_-,(\vec{g},\sigma),F_+)} is denoted by \s{[F_-,(\vec{g},\sigma),F_+]}.  See \cref{fig:expansion_of_tree_diagram} for an example.

\begin{figure}[h]
    \centering
\def\svgwidth{0.5\columnwidth}
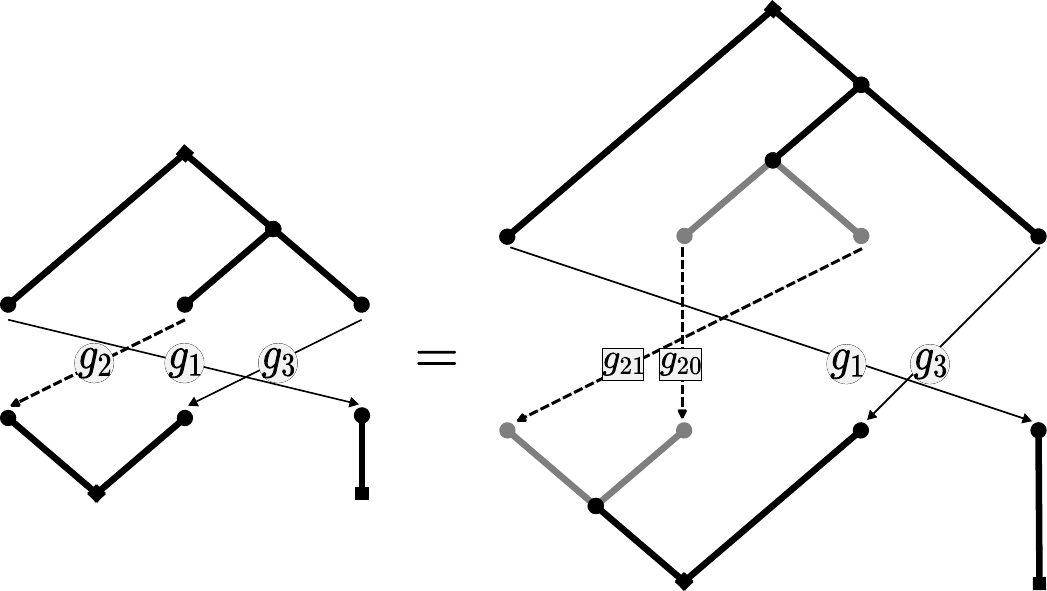
\caption{Two equivalent G-labeled paired forest diagrams where $\phi(g_2)=((g_{20},g_{21}),(1,2))$.}
\label{fig:expansion_of_tree_diagram}
\end{figure}

\begin{figure}[h]
    \centering
\def\svgwidth{\columnwidth}
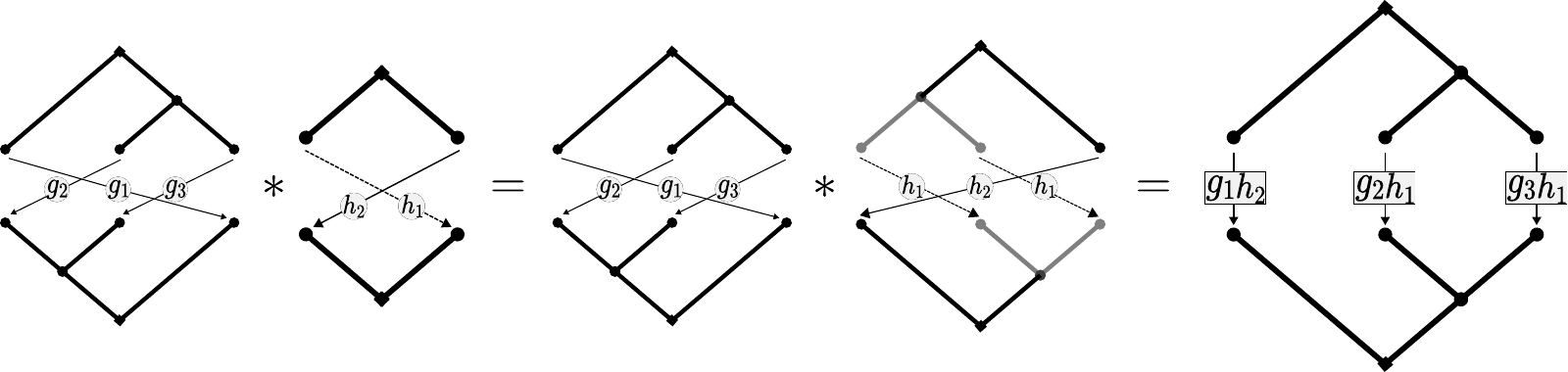
\caption{An example of the multiplication of two elements in $V(G)$.}
\label{fig:multiplication_of_tree_diagram}
\end{figure}
    
    We call a $G$-labeled paired forest diagram \emph{reduced} if no reduction can be performed. We check that every $G$-labeled paired forest diagram is equivalent to a unique reduced $G$-labeled paired forest diagram. Note that we need to use our assumption that the wreath recursion $\phi$ is injective. We can define a groupoid structure on the equivalence classes of $G$-labeled forest diagrams just as in all Thompson-like groups.  In fact, given \s{[F_-, (\vec{g},\sigma), F_+], [F_-', (\vec{g},\sigma'), F_+']} such that $F_+$ and $F_-'$ have the same number of components. When $F_+ = F_-'$, we declare 
    $$[F_-, (\vec{g},\sigma), F_+] \ast [F_-', (\vec{g}',\sigma'), F_+'] = [F_-, (\vec{g},\sigma)(\vec{g}',\sigma'), F_+']$$
    where the multiplication of $(\vec{g},\sigma)$ and $(\vec{g}',\sigma')$ uses the group structure on $G \wr \mathrm{S}_n$. In general, $F_+$ is not necessarily equal to $F_-'$, but since they have the same number of components, we could always make them the same after finitely many expansions. Thus, we could choose different representatives of $[F_-, (\vec{g},\sigma), F_+]$ and $ [F_-', (\vec{g}',\sigma'),F_+']$ using expansions, and define the multiplication.  See \cref{fig:multiplication_of_tree_diagram}. We call the equivalences classes of $G$-labeled paired forest diagrams equipped with such multiplication structure \emph{the $G$-labeled Thompson groupoid  \s{\mathscr{F}(G,\phi)}}.

    Define further $\mathscr{F}_{m,n}(G,\phi)$ to be the full subgroupoid of $\mathscr{F}(G,\phi)$ consisting of $G$-labeled paired tree diagrams $(F_-,(\vec{g},\sigma), F_+)$ such that $F_-$ has $m$ components and $F_+$ has $n$ components.  Note that \s{\mathscr{F}_{n,n}(G,\phi)} is a group and $\mathscr{F}_{1,1}(G,\phi)$ 
 is the $\phi$-labeled Thompson group $V_\phi(G)$.  But when $m\neq n$, $\mathscr{F}_{m,n}(G,\phi)$ is just a set.

\subsection{Brown's criterion and discrete Morse theory} To determine the finiteness property of $V_\phi(G)$, we will apply Brown's criterion \cite{Br87}. However, we shall use a version that is combined with discrete Morse theory. Recall that given a piecewise Euclidean cell complex $X$, a map
$h$ from the set of vertices of $X$ to the integers is called a
\emph{height function}, if each cell has a unique vertex maximizing~$h$.   For $t\in\mathbb{Z}$, define $X^{\leq t}$ to be
the full subcomplex of~$X$ spanned by vertices~$x$ satisfying $h(x)\leq
t$.  Similarly, define $X^{< t}$ and $X^{=t}$.  The \emph{descending star} $\dst(x)$ of a
vertex $x$ is defined to be the open star of~$x$ in~$X^{\le h(x)}$.  The
\emph{descending link} $\dlk(x)$ of $x$ is given by the set of ``local
directions'' starting at $x$ and pointing into $\dst(x)$. More details can be found in Section \ref{subsec:conn-dlink}.  We will use the following version of Brown's criterion \cite{Br87} which is combined with discrete Morse theory; see, for example, \cite[Theorem 6.2]{BZ22}.

\begin{thm}\label{lemma:Morse-brown-cri}
  Let $\Gamma$ be a group acting cellularly on a contractible piecewise Euclidean cell complex $X$.  Let $h:  X \rightarrow \mathbb{R}$ be a $\Gamma$-equivariant height function. Given \s{m\ge 1},  then $\Gamma$ is of type $\mathrm{F}_{m}$ provided that the following conditions are satisfied: 

  \begin{enumerate}
      \item  For each $q \in \mathbb{R}$ the sublevel set $X^{\leq q}$ is $\Gamma$-cocompact. 

\item  Each cell stabilizer is of type $\mathrm{F}_{m}$. 

\item For each $n \leq m$ there exists $t \in \mathbb{R}$ such that for all $x \in X^{(0)}$ with $h(x) \geq t$ we have that the descending link $\dlk (x)$ is $n$-connected.  

  \end{enumerate}
\end{thm}

\subsection{The \s{G}-labeled Stein--Farley complex} We  build the \s{G}-labeled Stein-Farley complex that $V_\phi(G)$ acts on now. Note that the group $G \wr \mathrm{S}_n$ embeds in the groupoid $\mathscr{F}(G,\phi)$ by sending $(\vec{g},\sigma)$ to \s{[1_n,(\vec{g},\sigma), 1_n]}, where $1_n$ denotes the elementary forest consisting of only $n$ roots. With this embedding, we have a right action of $G \wr \mathrm{S}_n$ on \s{\mathscr{F}_{m,n}(G,\phi)} by right multiplication.   Denote the orbit of an element $\alpha\in \mathscr{F}_{m, n}(G,\phi)$ under this action by $[\alpha]$.  Abuse of notation,  we may sometimes denote the orbit of an equivalence class  $[F_-,(\vec{g},\sigma),F_+]$ also by $[F_-,(\vec{g},\sigma),F_+]$. Note that if $\alpha \in \mathscr{F}_{m,n}(G,\phi)$ and $\beta_1,\beta_2 \in \mathscr{F}_{n,l}(G,\phi)$ with $[\alpha \ast \beta_1]=[\alpha \ast \beta_2]$, then by the groupoid structure we have $[\beta_1] = [\beta_2]$. We will refer to this as \emph{left cancellation} which will be helpful later for example in determining the descending link of a vertex; see Section \ref{subsec:conn-dlink}. 

Now let  $\mathcal{P}(G,\phi) =\bigcup_{n=1}^\infty\mathscr{F}_{1, n}(G,\phi)/ (G \wr \mathrm{S}_n)$. We have a well-defined height function \s{h:\mathcal{P}(G,\phi)\ra \mn} by sending elements in $\mathscr{F}_{1, n}(G,\phi)/ (G \wr \mathrm{S}_n)$ to $n$.

There is also a poset structure on $\mathcal{P}(G,\phi)$. Given $x,y \in \mathcal{P}(G,\phi)$ with a representative $[T,(\vec{g},\sigma),F]$ of $x$, say that $x\leq y $ if there exists a forest $F'$ with $m$ leaves such that $y$ lies in the same orbit as $[T,(\vec{g},\sigma),F] \ast [F', ( \vec{1}_G,\mathrm{id}), 1_m]$ where $\vec{1}_G$ is the identity element in $G^m$. When $F'$ is not equal to $1_m$,  $y$ is called a \emph{splitting} of $x$. Note that $h(y)> h(x)$ in this case. It is easy to check that this is a partial ordering on $\mathcal{P}(G,\phi)$ and independent of the choice of representative of $x$. If the forest $F'$ happens to be elementary, we denote the relation by $x\preceq y$ and call $y$ an \emph{elementary splitting} of $x$. If $x\preceq y$ but $x\neq y$, then we write $x\prec y$. Note that $\preceq$ and $\prec$ are not transitive, though it is true that if $x\preceq z$ and $x\leq y\leq z$, then $x\preceq y$ and $y\preceq z$.

  Note that  \s{V_\phi(G)}  acts on \s{\bigcup_{n=1}^\infty\mathscr{F}_{1, n}(G,\phi)} by left multiplication and this action commutes with the right action of \s{G \wr \mathrm{S}_n} and (elementary) splitting, hence \s{V_\phi(G)}
  left acts on \s{\mathcal{P}(G,\phi)} and preserves the relations $\leq$ and $\preceq$. Moreover the height function $h$ is \s{V_\phi(G)}-invariant.

Let \s{ |\mathcal{P}(G,\phi)| } be the geometric realization of the poset \s{(\mathcal{P}(G,\phi),\leq)}. 
Define \s{ \mathcal{P}(G,\phi) ^{\leq n}:=\{x\in  \mathcal{P}(G,\phi) : h(x)\leq n\}}, similarly  $ \mathcal{P}(G,\phi) ^{=n}:=\{x\in  \mathcal{P}(G,\phi) : h(x)= n\} $ and  $\mathcal{P}(G,\phi) ^{\geq n}:=\{x\in  \mathcal{P}(G,\phi) : h(x)\geq n\}$. 
Then  $| \mathcal{P}(G,\phi)^{\leq n} |$ is the full subcomplex of \s{| \mathcal{P}(G,\phi) |} spanned by all the vertices \s{x} with \s{h(x)\leq n}.
 
 \begin{definition} [$G$-labeled Stein--Farley complex]
      Let \s{\stein} be the subcomplex of \s{ |\mathcal{P}(G,\phi)| } consisting of all simplicies \s{[x_0<\cdots<x_n]} satisfying \s{x_0\preceq x_n}.
 \end{definition}

Note that $X_{G,\phi}$ and $ |\mathcal{P}(G,\phi)|$ has the same vertex set. Similarly, define \s{\stein^{\leq n}: = | \mathcal{P}(G,\phi)^{\leq n} |\cap \stein}. The complexes \s{\stein}, \s{| \mathcal{P}(G,\phi)^{\leq n} |, \stein^{\leq n}} all admit a \s{V_\phi(G)}-action. Exactly the same proof as for \cite[Proposition 2.1]{BFM+16} shows that   any two elements in \s{ \mathcal{P}(G,\phi)} have a least upper bound and if they have a common lower bound, they have a greatest lower bound. In particular, we have the following:

\begin{lemma}
    \s{( \mathcal{P}(G,\phi) , \leq)} is directed, \ie, every two elements in \s{\mathcal{P}(G,\phi)} has a common upper bound with respect to the order \s{\leq}. Thus \s{| \mathcal{P}(G,\phi) |} is contractible.
\end{lemma}
      
In \s{\mathcal{P}(G,\phi)}, for $x \leq y$ define the closed interval $[x, y]:=\{z \mid x \leq z \leq y\}$, and $|[x, y]|$ to be its geometric realization. Similarly, define the open and half-open intervals $(x, y),(x, y]$ and $[x, y)$. Note that if $x \preceq y$, then the closed interval $[x, y]$ is a Boolean lattice  (0 means root, 1 means caret in an elementary forest), and so the simplices in the geometric realization fit together into a cube. The top of the cube is $y$ and the bottom is $x$. In order to describe the cube clearly, let us introduce some notation first. Fix an integer $n$ and let $J$ be a subset of $[n] = \{1,2,\cdots,n\}$, denote $F_J^{(n)}$ the elementary forest obtained from $1_n$ by attaching a caret at every $j\in J$.   Furthermore, let $F_j^{(n)} := F_{\{j\}}^{(n)}$ and $\lambda_J^{(n)}:= [F_J^{(n)}, (\vec{1_G},\mathrm{id}), 1_{n+|J|}]$.

Now we want to show that \s{\stein} admits a cubical structure preserved by \s{V_\phi(G)}. In fact, the cubical cells are of the form \s{[x,y]} for \s{x\preceq y},    
    and the function \s{h} on the vertices of \s{\stein} extends to a Morse function. 
Also, the intersection of cubes is either empty or is itself a cube; in fact if $[x, y] \cap[z, w] \neq \emptyset$, then $y$ and $w$ have a lower bound, and we get that $[x, y] \cap[z, w]=[\sup (x, z), \inf (y, w)]$
where \s{\sup (x, z)} is the least upper bound of \s{x,z} and \s{\inf (y, w)} is the greatest lower bound of \s{y,w}. Therefore we can regard \s{\stein} as a cube complex \cite[p. 112, Definition 7.32]{BH99} on which \s
{V_\phi(G)} acts cubically.

\begin{proposition}\label{cocompact}
    For each \s{n\geq 1}, \s{\stein^{\leq n}} is \s{V_\phi(G)}-cocompact.
\end{proposition}
\begin{proof}

Note first that for each $k\geq 1$, $V_\phi(G)$ acts transitively on $\mathcal{P}(G,\phi) ^{= k}$. In fact, given two elements $[T_1, a_1,F_1]$ and $[T_2,a_2,F_2]$ in $\mathcal{P}(G,\phi) ^{= k}$ we can first assume $F_1$ and $F_2$ are the same using the equivalence relation. Then $[T_1, a_1,F_1] = [T_1, a_1a_2^{-1}, T_2]\ast [T_2,a_2,F_1]$. Thus there exists for each $1\leq k \leq n$ one orbit of vertices $x \in \stein^{\leq n}$ with $h(x)=k$. Given a vertex $x$ with $h(x)=k$, there are only finitely many cubes $C_1,\cdots, C_r$  in the sublevel set $\stein^{\leq n}$ that has $x$ as their bottom. If $C$ is a cube in $\mathcal{P}(G,\phi) ^{\leq n}$ such that its bottom is in the same orbit as $x$, then the cube $C$ must be in the same orbit as $C_i$ for some $1\leq i\leq r$.  It follows that there can only be finitely many orbits of cubes in \s{\stein^{\leq n}}. 
\end{proof}

\begin{proposition} \label{prop:inter-cont}
    For \s{x\leq y} with $x \not \prec y $, the complex \s{|(x, y)|} is contractible.
\end{proposition}
\begin{proof}
The proof uses Quillen's conical contraction \cite[Section 1.5]{Qui78} just as \cite[Lemma 2.3]{BFM+16} or \cite[Lemma 4.4]{SW23}. For that we need to define a function $c:(x,y) \to (x,y)$ such that there is $y_0$ in $(x,y)$ such that $z\geq c(z)\leq y_0$. In fact, for any $z\in (x,y)$, let $c(z)$ be the largest element of $(x,z)$ satisfying $x\preceq c(z)$, and let $y_0=c(y)$. One checks that $c$ is well-defined and the condition  $z\geq c(z)\leq y_0$ is satisfied for any $z\in (x,y]$. Therefore \s{|(x, y)|} is contractible.
\end{proof}

\begin{proposition}
    \s{\stein} is homotopy equivalent to \s{\pg}, hence \s{\stein} is also contractible. 
\end{proposition}
\begin{proof}
We again follow the standard proof here via building up $  \pg  $  from  $\stein$ by attaching new subcomplexes in such a way that the homotopy type stays the same. Given a closed interval $[x,y]$, define $r([x,y]) = h(y)-h(x)$. We attach the contractible complexes $|[x,y]|$ for $x\not \npreceq y$ to $X_{G,\phi}$ in increasing order of $r$ value, attaching $|[x,y]|$ along $|[x,y)\bigcup (x,y]|$ which is a suspension of $ (x,y)$. But $(x,y)$ is contractible by  \Cref{prop:inter-cont}. So the attaching process does not change the homotopy type of the complexes. As being contractible is closed under taking union, we have \s{\stein} is homotopy equivalent to \s{\pg}.
\end{proof}

We proceed to analyze the finiteness property of the stabilizers. We remind the reader that an element lies in the stabilizer of a cube if it maps the cube to itself. Since a cube has only finitely many vertices, its stabilizer has a finite index subgroup fixing the cube pointwise.

\begin{lemma}\label{lemma:v-stab}
    Let \s{x\in \mathcal{P}(G,\phi)} with \s{h(x)=n}, then the stabilizer $\mathrm{Stab}_{V_\phi(G)}(x)$ of \s{x} in \s{V_\phi(G)} is isomorphic to \s{G \wr \mathrm{S}_n}.
\end{lemma}
\begin{proof}
    Fix $\alpha \in \mathscr{F}(G,\phi)$ with $[\alpha]= x$ and suppose $f \in V_\phi(G)$ fixes $x$. Then $[f \ast \alpha] = [\alpha]$ which means $\alpha^{-1} \ast f\ast \alpha = [1_n,(\vec{g},\sigma),1_n)]$ for some $(\vec{g},\sigma) \in  G\wr S_{n}$. We can further map elements of this form to $G\wr \mathrm{S}_n$ by sending $[1_n,(\vec{g},\sigma),1_n)]$ to $(\vec{g},\sigma)$, hence we have a map from $\mathrm{Stab}_{V_\phi(G)}(x)$ to $ G\wr S_{n}$. It has an inverse  $(\vec{g},\sigma) \mapsto \alpha \ast [1_n,(\vec{g},\sigma),1_n)] \ast \alpha^{-1}$. Obviously, this map is a group homomorphism hence an isomorphism from $\mathrm{Stab}_{V_\phi(G)}(x)$ to $G \wr \mathrm{S}_n$.
\end{proof}

\begin{lemma}\label{lemma:cell-stab}
    Let $[x,y]$ be a cube in $X_{G,\phi}$ with $h(x) = n$. Then the stabilizer of $[x,y]$ is a finite index subgroup of $G \wr \mathrm{S}_n$. In particular, if $G$ is of type $F_m$, so is the stabilizer of $[x,y]$. 
\end{lemma}
\begin{proof}

First observe that an element of $V_\phi(G)$ fixes the cube $[x,y]$ if and only if it fixes both $x$ and $y$. Suppose $x = [\alpha]$, $h(x) =n$ and $f$ fixes $x$. By \Cref{lemma:v-stab}, we first have $\alpha^{-1}\ast f\ast \alpha $ lies in the subgroup of \s{\mathscr{F}_{n,n}(G,\phi)} which isomorphic to  $G \wr \mathrm{S}_n$.  Since $x \prec y$, there exists an elementary forest $F_J^{(n)}$ such that $y = [\alpha \ast \lambda_J^{(n)}]$, where $J$ is a subset of $[n]$. Now the element $f$ also fixes $y$, so $[f\ast \alpha \ast \lambda_J^{(n)}] = [\alpha \ast \lambda_J^{(n)}]$, which means $[\alpha^{-1} \ast f\ast \alpha \ast \lambda_J^{(n)}] =  [\lambda_J^{(n)}]$. If we write $\alpha^{-1} \ast f\ast \alpha = (\vec{g},\sigma)$,  to make $[\alpha^{-1} \ast f\ast \alpha \ast \lambda_J^{(n)}] =  [\lambda_J^{(n)}]$, we just need  $\sigma(J) = J$. So the stabilizer of $[x,y]$  can be identified with the subgroup $\{(\vec{g},\sigma) \in G \wr \mathrm{S}_n \mid \sigma(J) = J\}$ which is a finite index subgroup of $G \wr \mathrm{S}_n$. 
\end{proof}

\subsection{Connectivity of the descending links} \label{subsec:conn-dlink} Let us give a description of the descending link first.
Given a vertex $x$ in $\mathcal{P}(G,\phi)$. The \text{descending star} $\dst$ of $x$ in $X_{G,\phi}$ consists of all the cubes of the form $[y,x]$ where $y\preceq x$. For such a cube $C=[y,x]$ let $\mathrm{bot}(C) :=y$ be the map giving the bottom vertex. This is a bijection from the set of such cubes to the set $D(x):= \{y \in \mathcal{P}(G,\phi) \mid y\preceq x\}$. The cube $[y',x]$ is a face of $[y,x]$ if and only if $y'\in [y,x]$, if and only if $y'\geq y$. Thus a cube $C'$ is a face of $C$ if and only if $\mathrm{bot} (C') \geq \mathrm{bot}(C)$ and $\mathrm{bot}$ is an order reversing map. Therefore, a simplex in $\dlk (x)$ can be identified with an element $y$ in $\mathcal{P}(G,\phi)$ such that $y\prec x$, where the dimension of the simplex plus one is just the number of elementary splits needed to get from $y$ to $x$. And face relation is reverse of the relation $<$ on $D(x) \setminus x$. Since $X_{G,\phi}$ is a cube complex, $\dlk(x)$ is a simplicial complex.

Now for each vertex $x\in \mathcal{P}(G,\phi)$ with $h(x) =n$ and a simplex in the descending link corresponding to a cube $[y,x]$ with $x$ as top, we pick representatives $y_1$ and $x_1$ in $\mathscr{F}(G,\phi)$ for $y$ and $x$. Then there must exist an elementary forest $F_J^{(n-|J|)}$ and $(\vec{g},\sigma)$ in $G \wr \mathrm{S}_n$ such that $x_1=y_1\ast [F_J^{(n-|J|)}, (\vec{1_G}, id) , 1_n ]\ast [1_n, (\vec{g},\sigma), 1_n]=y_1\ast [F_J^{(n-|J|)},(\vec{g},\sigma), 1_n]$. So every bottom vertex $y$ has the form $y=[y_1]=[x_1\ast [1_n,(\vec{g},\sigma), F_J^{(n-|J|)}]]$ and every vertex of the form of right hand side gives a bottom vertex.  By left multiplying with $x_1^{-1}$ (or use left cancellation), $\dlk(x)$ is isomorphic to the following simplicial complex: $[1_n, (\vec{g},\sigma), F_J^{(n-|J|)}]$ for any $J$ nonempty forms a $(|J|-1)$-simplex, and face relation given by reversing of the ordering $\leq$ in $\mathcal{P}(G)$. In particular,  $\dlk(x)$ is a flag complex. 

%Now if $h(x) =n$, by left cancellation, $\dlk(x)$ is isomorphic to the following simplicial complex $\mathcal{E}_n (G,\phi)$: $[1_n, (\vec{g},\sigma), F_J^{(n-|J|)}]$ for $J$ nonempty, and face relation given by reversing of the ordering $\leq$ in $\mathcal{P}(G,\phi)$. 

Note that when $G$ is trivial, we denote the descending link by \s{\mathcal{E}_n(1)}. The simplices of \s{\mathcal{E}_n(1)} are of the form $[1_n, \sigma, F_J^{(n-|J|)}]$. But up to the equivalence relation on $\mathscr{F}_{n,n-|J|}(G,\phi)$ (i.e. right action by $S_{n-|J|}$), a vertex in $\mathcal{E}_n(1)$ is a pair $(r_1,r_2)$ (order matters here) of two different roots in $1_n$, and $k+1$ vertices form a $k$-simplex if the roots of the vertices are pairwise disjoint. This complex is closely related to the matching complex $\mathcal{M}_n$, see for example \cite[Section 3.2]{BFM+16}. Recall that $\mathcal{M}_n$ can be defined as follows: a vertex in $\mathcal{M}_n$ is a subset $\{r_1,r_2\}$ (order does not matter here) of two roots in $1_n$, and $k+1$ vertices form a $k$-simplex if they are pairwise disjoint. There is a natural surjection map from $u: \mathcal{E}_n(1) \to \mathcal{M}_n$ by forgetting the order. We are now ready to define a forgetful map $\pi:\mathcal{E}_n (G,\phi)\ra \mathcal{M}_n$ by defining it on vertices: given a vertex in $\mathcal{E}_n (G,\phi)$, we first represent it by $[1_{n}, (\vec{g},\sigma), 
 F ]$ in $\mathscr{F}_{n,n-|J|}(G,\phi)$,  map it first to $[1_{n}, \sigma, 
 F ]$, then to $u([1_{n}, \sigma, 
 F ])$. One checks that $\pi$ is well defined. Note that  the map $\mathcal{E}_n (G,\phi)\ra \mathcal{E}_n(1)$ sending $[1_{n}, (\vec{g},\sigma), 
 F ]$ to $[1_{n}, \sigma, 
 F ]$ is actually not well-defined  due to the right action of $G\wr S_n$ on $\mathscr{F}_{n,n-|J|}(G,\phi)$.

In summary, the forgetful map $\pi$ from $\mathcal{E}_n (G,\phi)$ to $\mathcal{M}_n$ maps representative $[1_{n}, (\vec{g},\sigma), 
 F ]$  to the collection of pairwise disjoint subsets:
 $$\{\{i,j\}\subseteq [n]:\text{there exist a caret in } F \text{ such that } (i)\sigma,(j)\sigma \text{ are leaves of this caret}\}.$$
 Since the positions of carets do not affect the map $\pi$,  this map is invariant under the right action of $G\wr S_{n-|J|}$ on representatives, hence it is well-defined. And this map is also a simplicial map since the face relation in $\mathcal{M}_n$ is the inclusion relation of sets. Also, this map is surjective by noting that the map $u$ is surjective.

We want to show that the map $\pi$ is a complete join so we can apply \cite[Proposition 3.5]{HW10} to show that  $\mathcal{E}_n (G,\phi)$ is highly connected. Let us recall the definition here.

\begin{definition}
A surjective simplicial map $\pi: Y\to X$ is called a \emph{complete join} if it satisfies the following properties:
\begin{enumerate} [label=(\arabic*)]
    \item $\pi$ is injective on individual simplices. 
    \item For each $p$-simplex $\sigma = \langle v_0,\cdots,v_p\rangle$ of $X$, $\pi^{-1}(\sigma)$ is the join $\pi^{-1}(v_0)\ast \pi^{-1}(v_1)\ast\cdots\ast \pi^{-1}(v_p)$.
 \end{enumerate}
\end{definition}

\begin{lemma}\label{lem:join}
    Let $[1_{n}, (\vec{g}_1,\sigma_1),   F_{i}^{(n-1)} ]$ and $[1_{n}, (\vec{g}_2,\sigma_2),   F_{J}^{(n-k-1)} ]$ be two simplicies of dimension \s{0} and \s{k} in \s{\mathcal{E}_n (G,\phi)}. Suppose that their images under the map $\pi$ are contained in a simplex of $\mathcal{M}_n$. Then there exists a simplex in $\mathcal{E}_n (G,\phi)$ containing both of them.
\end{lemma}
\begin{proof}

   We may assume that $[1_{n}, (\vec{g}_1,\sigma_1),   F_{i}^{(n-1)} ]$ is not a face of $[1_{n}, (\vec{g}_2,\sigma_2),   F_{J}^{(n-k-1)} ]$. 
   There is an action of $G \wr \mathrm{S}_n$ on $\mathcal{E}_n (G,\phi)$ given by
$$
(\vec{g}',\sigma')[1_{n}, (\vec{g},\sigma), 
 E ]=[1_{n}, (\vec{g}',\sigma')(\vec{g},\sigma), 
 E ].
$$
Recall first that by definition $[1_{n}, (\vec{g},\sigma), 
 E ]$ represents the equivalence class under the right action of $G \wr S_{n-|J|}$. One checks now for each $k \geq 0$, the action is transitive on the $k$-simplices of $\mathcal{E}_n (G,\phi)$. We can therefore assume without loss of generality that $[1_{n}, (\vec{g}_2,\sigma_2),   F_{J}^{(n-k-1)}]=[1_{n}, ( \vec{1}_G,\mathrm{id}), F_{[k+1]}^{(n-k-1)}]$ where \s{F_{[k+1]}^{(n-k-1)}} is the forest with the first \s{k+1} components being  carets and the other \s{n-2k-2} components all single roots.

Since the image of $[1_{n}, ( \vec{1}_G,\mathrm{id}),   F_{[k+1]}^{(n-k-1)} ]$ under $\pi$ is $\{\{1,2\},\{3,4\},\cdots, \{2k+1,2k+2\}\}$,
and the image of $[1_{n}, (\vec{g}_1,\sigma_1),   F_{i}^{(n-1)} ]$ and $[1_{n}, ( \vec{1}_G,\mathrm{id}),   F_{[k+1]}^{(n-k-1)} ]$ under $\pi$ are different, then $\sigma^{-1}_1(i),\sigma^{-1}_1(i+1)$ must be disjoint from \s{\{1,2,\cdots, 2k+2\}}.
%Since the image of $[1_{n}, (\vec{g}_1,\sigma_1),   F_{i}^{(n-1)} ]$ and $[1_{n}, ( \vec{1}_G,\mathrm{id}),   F_{[k]}^{(n-k-1)} ]$ under $\pi$ are different,  $\sigma^{-1}_1(i),\sigma^{-1}_1(i+1)$ must be disjoint from \s{\{1,2,\cdots, 2k+2\}}.
Up to equivalence (i.e. right action of $G\wr S_{n-1}$) we can also assume that \s{i,i+1} are disjoint from \s{\{1,2,\cdots, 2k+2\}} and the first $2k+2$ coordinates of $\vec{g}_1$ are $1_G$, in particular we can assume that \s{\sigma_1} leaves \s{\{1,2,\cdots, 2k+2\}} vertex-wise fixed. Now it is easy to see that \s{[1_n,(\vec{g}_1,\sigma_1), F_{[k+1]}^{(n-k-1)}]=[1_n, (\vec{1}_G,\mathrm{id}), F_{[k+1]}^{(n-k-1)}]}, so \s{[1_n,(\vec{g}_1,\sigma_1), F_{[k+1]\cup \{i-k-1\}}^{(n-k-2)}]} is a simplex containing $[1_n, (\vec{1}_G,\mathrm{id}), F_{[k+1]}^{(n-k-1)}]$ and \s{[1_n,(\vec{g}_1,\sigma_1), F_i^{(n-1)}]} as faces.
\end{proof}

\begin{lemma}\label{lem:comp-join}
    \s{\pi} is a complete join.
\end{lemma}
\begin{proof}
    By definition, it is easy to see that the map $\pi$ is surjective and injective on individual simplices. Let $e$ be a $k$-simplex in $\mathcal{M}_n$ with vertices $v_{0}, \ldots, v_{k}$. To prove $\pi$ is a complete join, it just remains to show

$$
\pi^{-1}(e)=\underset{0\leq j\leq k}{*}\pi^{-1} (v_{j} )
$$

`` $\subseteq$ ": For every $k$-simplex in $\pi^{-1}(e)$, it's different vertices lie in different $\pi^{-1} (v_{j} )$ for $0\leq j\leq k$. Hence, all such $k$-simplex is contained in the right hand side.
%This inclusion just says that vertices in $\pi^{-1}(e)$ that are connected by an edge map to distinct vertices under $\pi$ which is clear.

``$\supseteq$ ": We prove this by induction on $k$. Suppose $\pi^{-1}(e) \supseteq \underset{0\leq j\leq k}{*} \pi^{-1} (v_{j} )$ for $k=r$. Now given an $(r+1)$-simplex $e'= \langle v_{0}, \cdots, v_{r+1} \rangle$ which is a join of $e= \langle v_{0}, \cdots, v_{r} \rangle$ and $v_{r+1}$. We just need to show for any simplex $\bar{e}, \bar{v}_{r+1} \in \mathcal{E}_n (G,\phi)$ such that $\pi(\bar{e})=e, \pi (\bar{v}_{r+1})=v_{r+1}$, we have a $(r+1)$-simplex that contains both $\bar{e}$ and $v_{r+1}$. But this follows immediately from \Cref{lem:join} so we complete the proof.
\end{proof}

\begin{corollary}\label{connect}
     $\mathcal{E}_n(G,\phi)$ is  \s{(\lfloor\frac{n+1}{3}\rfloor-2) } -connected.
\end{corollary}
\begin{proof}
The matching complex $\mathcal{M}_n$ is \s{(\lfloor\frac{n+1}{3}\rfloor-2) }-connected, see for example \cite[Proposition 3.6]{BFM+16}. And observe that the link of a $k$-simplex is isomorphic to  $\mathcal{M}_{n-2k}$ , which is \s{(\lfloor\frac{n-2k+1}{3}\rfloor-2) } -connected. Hence \s{\mathcal{M}_n} is wCM of dimension \s{(\lfloor\frac{n+1}{3}\rfloor-1) }. Recall that a simplicial complex  is wCM of dimension $n$ if it is $(n-1)$-connected and the link of each $k$-simplex is $(n-k-2)$-connected.  Now by \cite[Proposition 3.5]{HW10}, \s{\mathcal{E}_n (G,\phi)} is also wCM of dimension \s{(\lfloor\frac{n+1}{3}\rfloor-1) }, in particular, it is \s{(\lfloor\frac{n+1}{3}\rfloor-2) } -connected.
\end{proof}

\begin{thm}
    $X_{G,\phi}$ is an infinite dimensional complete CAT(0) cube complex.
\end{thm}
\proof 
We have seen that $X_{G,\phi}$ is a cube complex. We first show that $X_{G,\phi}$ is CAT(0). Note that the dimension of cubes in $X_{G,\phi}$ is not bounded.  So by Gromov's link condition \cite[Appendix B.8]{Lea13} for infinite dimensional CAT(0) cube complex, it suffices to show that the link at each vertex $X_{G,\phi}$ is a flag complex.

Given a vertex $x$ in $X_{G,\phi}$, the vertices in the link $\lk(x)$ either lie in the descending link $\dlk(x)$ or the ascending link $\alk(x)$. Let us represent $x$ by $[T, (\vec{g},\sigma), F]$ where the tree $T$ and the forest $F$ have the same number of leaves, and $F$ has $n$ feet.  Now using the isomorphism between $\dlk(x)$ and $\mathcal{E}_n(G,\phi)$ and the map $\pi$, we can associate with each vertex in $\dlk(x)$ a pair of trees in $F$. By Lemma \ref{lem:comp-join}, given any $k+1$ vertices in $\dlk(x)$, they form a $k$-simplex if and only if their associated tree pairs in $F$ are pairwise disjoint. Similarly, a vertex in $\alk(x)$ can be obtained by performing  an elementary splitting at some tree in $F$. So a vertex in $\alk(x)$ can be identified with a tree in $F$. And any $k+1$ trees in $F$ form a $k$-simplex in $\alk(x)$. In particular, we see that $\alk(x)$ is an $(n-1)$-simplex.  Notice now that for a $k_1$-simplex in $\dlk(x)$ and a $k_2$-simplex in $\alk(x)$, they form a $(k_1+k_2+1)$-simplex if and only if the trees associated with them are pairwise disjoint. Here we use the fact that the top vertex of a cube in $X_{G,\phi}$ can be obtained from its bottom vertex via an elementary splitting. Therefore, $\lk(x)$ is a flag complex.

Now by \cite[Theorem A.6]{Lea13}, to show that $X_{G,\phi}$ is complete, we only need to prove that every ascending sequence of cubes terminates. But given a vertex $x$ in $X_{G,\phi}$, the maximal dimension of cubes containing $x$ is bounded by its height which is finite. Hence, every ascending sequence of cubes must terminate.

\qed

\begin{corollary}
    $V_{\phi}(G)$ does not have property (T).
\end{corollary}
\proof 
By our construction so far, $V_{\phi}(G)$ acts on the infinite-dimensional complete CAT(0) cube complex $X_{G,\phi}$. By \Cref{lemma:cell-stab}, the stabilizers are finite index subgroups of groups of the form $G\wr S_n$. In particular,  $V_{\phi}(G)$ acts on $X_{G,\phi}$ without a global fixed point. This further implies that $V_{\phi}(G)$ acts on $X_{G,\phi}$ with unbounded orbits \cite[II.2.7]{BH99}.   Hence $V_{\phi}(G)$ cannot have property (T) \cite[Theorem 1.2]{CDH10}. 
\qed
\begin{remark}
When the corresponding action of $G$ on the binary tree is trivial, one has a surjection from $V_\phi(G) \to V$. And one can conclude that $V_{\phi}(G)$ does not have property (T) from the fact that $V$ already does not have this \cite{Fa03b}. But once the action is nontrivial, it might be hard for  $V_\phi(G)$ to surject to $V$. For example, take $G$ to be  Grigorchuk’s first group, the corresponding group is R\"over’s group and is simple \cite{Ro99}. 
\end{remark}

\begin{proof}[Proof of Theorem \ref{finiteness-main}]
    We apply \Cref{lemma:Morse-brown-cri}. Consider the action \s{V_\phi(G)} on the cube complex \s{X_{G,\phi}} with the height function $h$. By \cref{cocompact}, the action of $V_\phi(G)$ on $X_{G,\phi}$ is cocompact on each skeleton;  by \Cref{lemma:cell-stab}, the stabilizer of a cell \s{[z,y]} is of type \s{F_m} since $G$ is of type $F_m$; and by \cref{connect} the high connectivity of the descending link of a vertex  is guaranteed when its height is big enough. So by \cref{lemma:Morse-brown-cri}, \s{V_\phi(G)} is of type \s{F_m}.
\end{proof}

Combining Theorem \ref{finiteness-main} with \Cref{cor:fin-VG-imply-G}, we have:

\begin{corollary} \label{thm:fin-VG}
    \s{V(G)} is of type \s{F_m} if and only if \s{G} is of type \s{F_m}.
\end{corollary}

 The  Corollary \ref{thm:fin-VG} does not hold for arbitrary injective wreath recursions, see Remark \Cref{rem:RN_gp}. In general, given an injective wreath recursion of the form  $\phi:G\to G\times G$ (in this case of  action of $G$ on the  infinite rooted binary tree is trivial), Tanushevski also defines $\phi$-labeled Thompson group $F_{\phi}(G)$ and $T_{\phi}(G)$ and studies finiteness properties of $F_\phi(G)$. Given a presentation of $G$, he  determine a presentation of $F_{\phi}(G)$ \cite{Tan16}. In particular,
he proves that $F_\phi(G)$ is finitely generated (resp. finitely presented) if G is finitely generated (resp. finitely
presented).  We record the following improvement of his  result.

    \begin{thm}\label{thm-fin-FT-G}
        Let $\phi$ be an injective wreath recursion of the form $\phi:G\to G\times G$. If $G$ is of type $F_n$, so are $F_\phi(G)$ and $T_\phi(G)$. 
    \end{thm}
\begin{proof}[Sketch of Proof]
    As the proof is parallel to the $V_\phi(G)$ case, we only give a sketch. One first defines the \emph{cyclic $G$-labeled paired forest  diagram} (resp. ordered $G$-labeled paired forest  diagram) to be a $G$-labeled paired forest  diagram of the form $(F_-,(\vec{g},\sigma),F_+)$ where $\sigma$ is a cyclic (resp. trivial) permutation. Then one defines the corresponding groupoids and complexes. The only essential difference then appears in the analysis of the connectivity of the descending links. But one again uses a complete join argument, reducing the connectivity of descending links to that of cyclic matching complexes (resp. linear matching complexes). The connectivity of these complexes is well known; see, for example, \cite[Lemma 3.19]{Wi19} (resp. \cite[Proposition 11.16]{Koz08}).    
\end{proof}

\section{Tools to compute the bounded cohomology of groups}
 We collect some tools that can be used to show that a group is boundedly acyclic in this section. The readers should consult \cite{Fri17} for the basics of bounded cohomology. The following cohomological techniques are developed in \cite{LLM22, Mo22, MN23, MR23} to compute the bounded cohomology of groups.

\begin{definition}\label{generic}
Let $X$ be a set.  We call a binary relation $\perp$ \emph{generic} if  for every given finite set $Y \subseteq X$,  there is an element $x \in X$ such that $y \perp x$ for all $y \in Y$. 
\end{definition}
Recall that any generic relation on $X$ gives rise to a semi-simplicial set $X_{\bullet}^{\perp}$ in the following way:  we define $X_{n}^{\perp}$ to be the set of all $(n+1)$-tuples $\left(x_{0},  \ldots,  x_{n}\right) \in X^{n+1}$ for which $x_{i} \perp x_{j}$ holds for all $0 \leq j<i \leq n$.  The face maps $\partial_{n}\colon  X_{n}^{\perp} \rightarrow X_{n-1}^{\perp}$ are the usual simplex face maps,  \ie,
$$
\partial_{n}\left(x_{0},  \ldots,  x_{n}\right)=\sum_{i=0}^{n}(-1)^{i}\left(x_{0},  \ldots,  \widehat{x}_{i},  \ldots,  x_{n}\right)
$$
where $\widehat{x}_{i}$ means that $x_{i}$ is omitted. 

\begin{lemma}\cite[Proposition 3.2]{MN23}\label{lemma:generic-ba}
    If the binary relation $\perp$ is generic, then the complex $X_{\bullet}^{\perp}$ is boundedly acyclic (and connected).
\end{lemma}

\begin{proposition}\cite[Theorem 3.3 + Remark 3.4]{MN23}\label{prop:BA-crit}
     Let $\Gamma$ be a group acting on a boundedly acyclic connected semi-simplicial set $X_\bullet$. Suppose further that for every $n\ge0$ we have:
\begin{enumerate}[label=(\roman*)]
    \item The stabilizers of any point in $X_{n}$ are boundedly acyclic.
    \item The  \s{\Gamma}-action on $X_{n}$ has finitely many orbits.
   % \item Given $q>0$ with $p+q<N$, the q-th vanishing moduli of all those stabilisers are uniformly bounded.
\end{enumerate}
Then there is an isomorphism $\mathrm{H}_{b}^{p}(\Gamma,\mathbb{R}) \cong \mathrm{H}_{b}^{p}\left(X_{\bullet} / \Gamma,\mathbb{R}\right)$ for every $p\geq 0$.
\end{proposition}

Combining Lemma \ref{lemma:generic-ba} and Proposition \ref{prop:BA-crit}, we have the following corollary.

\begin{corollary}\label{cor:BA-gene}
    Let $\Gamma$ be a group acting on a set $X$ and there is a generic relation $\perp$ on $X$ preserved by $\Gamma$.  If for every $n \geq 0$,  the action of \s{\Gamma} on \s{X^\perp_n} is transitive  and the stabilizers are boundedly acyclic, then \s{\Gamma} is boundedly acyclic.
\end{corollary}
\begin{proof}
    From Lemma \ref{lemma:generic-ba}, We have $X_{\bullet}^{\perp}$ is boundedly acyclic and connected. Since $\Gamma$ is transitive on each $X_n^{\perp}$ and all stabilizers are boundedly acyclic, by proposition \ref{prop:BA-crit}, we get
    $$\mathrm{H}_{b}^{p}(\Gamma,\mathbb{R}) \cong \mathrm{H}_{b}^{p}\left(X_{\bullet}^{\perp} / \Gamma,\mathbb{R}\right),\forall p\geq0.$$
    Moreover we know
    $$\ell^\infty\left(X_{n}^{\perp} / \Gamma\right)\cong\mathbb{R}\quad\text{and}\quad\delta_n=\begin{cases}
        \text{id}_\mathbb{R} & 2\not\mid n\\
        0 & 2\mid n
    \end{cases},$$
   
    This immediately implies that  $X_{\bullet}^{\perp} / \Gamma$ is boundedly acyclic, so is $\Gamma$ . 
\end{proof}
  
%Thus to compute the bounded cohomology of \s{G}, we can define a \s{G}-set \s{X} where the diagonal \s{G}-action on \s{X^{n+1}} has finitely many orbits, every stabilizer is bounded acyclic and a generic relation \s{\perp} on \s{X} preserved by \s{G}.  

We also need the following criterion to prove the boundedly acyclicity for stabilizers.

\begin{thm}\label{coame-to-wreath} \cite[Corollary 5]{Mo22}
    Let $\Gamma$ be a group and $\Gamma_{0}<\Gamma$ a subgroup with the following two properties: 
    \begin{enumerate}
        \item Every finite subset of $\Gamma$ is contained in some $\Gamma$-conjugate of $\Gamma_{0}$ , 
        \item $\Gamma$ contains an element $g$ such that the conjugates of $\Gamma_{0}$ by $g^{p}$ and by $g^{q}$ commute for all $p \neq q$ in $\mathbb{Z}$ . 
    \end{enumerate}
Then $\mathrm{H}_{\mathrm{b}}^{n}(\Gamma,  \mathbb{R})=0$ holds for all $n>0$. 
\end{thm}

We also use the following powerful generalization of  \cref{coame-to-wreath} to prove bounded acyclicity.

\begin{definition}\cite[Definition 1.1]{CFLM23}\label{def of cc}
     A group $\Gamma$ has \emph{commuting cyclic conjugates} if for every finitely generated subgroup $H \leq \Gamma$ there exist $t \in \Gamma$ and $k \in \mathbb{N}_{\geq 2} \cup\{\infty\}$ such that:
\begin{enumerate}
    \item $\left[H, t^{p} H t^{-p}\right]=1$ for all $1 \leq p<k$ ,
    \item $\left[H, t^{k}\right]=1$ .
\end{enumerate}
Here we read that $t^{\infty}=1$ .
\end{definition}

\begin{thm}\cite[Theorem 1.3]{CFLM23}\label{commuting conjugates}
    Let $\Gamma$ be a group with commuting cyclic conjugates. Then  \s{\Gamma} is boundedly acyclic.
\end{thm}

%\textcolor{green}{Note that \cite[Theorem 1.3]{CFLM23} actually showed that \s{\text{H}^n_\mathrm{b}(\Gamma, E)=0, \forall n\geq 1} for  every separable dual
%Banach \s{\Gamma}-module \s{E}, following the spirit of \cite[Corollary 5]{Mo22}, which proved the bounded acyclicity with separable dual
%Banach \s{\Gamma}-module coefficients, not just \s{\mr}.} 
Recall that a subgroup \s{H\leq \Gamma} is \emph{co-amenable} in \s{\Gamma} if there is a \s{\Gamma}-invariant mean on \s{\Gamma/H}. Since \s{\Gamma/[\Gamma,\Gamma]} is abelian hence amenable, \s{[\Gamma,\Gamma]} is co-amenable in \s{\Gamma}. The following lemma is proved in \cite[Proposition 8.6.6]{Mon01}, see also \cite{MP03}.

\begin{lemma}\label{Lem:coame-bcohg}
    Let $\Gamma$ be a group, $\Gamma_{0}<\Gamma$ a co-amenable subgroup. Then the restriction map
$$
\mathrm{H}_{\mathrm{b}}^{\bullet}(\Gamma, \mathbb{R}) \longrightarrow \mathrm{H}_{\mathrm{b}}^{\bullet}\left(\Gamma_{0}, \mathbb{R}\right)
$$
is injective.
\end{lemma}

\begin{remark}
Some of the results we cited here work for more general coefficients. For example, \Cref{coame-to-wreath}, \Cref{commuting conjugates} and \Cref{Lem:coame-bcohg} work for all separable dual Banach $G$-modules. But the proof of \Cref{prop:BA-crit} uses a spectral sequence argument, so \Cref{cor:BA-gene} and it are only available for constant real coefficients so far. It is also worth mentioning that after our paper came out, Bogliolo \cite{Bo25} proves new boundedly acyclic results with respect to $\mathfrak{X}^{\text {ssep }}$ coefficients. She then uses this to obtain some new embedding results \cite[Theorem D]{Bo25} (see also \cite[Remark 8.3]{Bo25}). In contrast with our embedding results,  her host groups will have many finite index subgroups.
\end{remark}

\section{Bounded acyclicity of  $\phi$-labeled Thompson groups}\label{lable ba}

In this section, we prove that the $\phi$-labeled Thompson groups are all boundedly acyclic  based on ideas in \cite{An22,Mo22,MN23}.

Given a group \s{G} with wreath recursion \s{\phi}, denote the corresponding labeled Thompson group by $V_\phi(G)$. By \Cref{thm:inj-iso}, we can assume from now on that $\phi$ is injective. 
A \emph{dyadic interval} is a subset of \s{\mc = \{0,1\}^\omega} of the form \s{u\xinf} for some \s{u\in \xstar}. Recall that  dyadic intervals are clopen and the set of dyadic intervals form a topological basis for \s{\mc}.  A \emph{dyadic neighborhood} of \s{x\in \mc} is a dyadic interval containing \s{x}.

\begin{definition}
     Given a $G$-matrix
    $$\mathcal{M}=\begin{scriptsize}\begin{pmatrix}  
u & u_2 & \cdots & u_n \\
g & g_2 & \cdots &g_n \\
v& v_2 & \cdots &v_n
\end{pmatrix}  \end{scriptsize} $$ we define its \emph{label} at the dyadic interval $U=u\xinf$ to be $g$ . For any element $\alpha\in V_{\phi}(G)$, we define the \emph{label} of $\alpha$ at $U$ is $l(_U|\alpha)=g$ if there is a $G$-matrix $\mathcal{M}$ with $[\mathcal{M}]=\alpha$ and \emph{label of $\mathcal{M}$} at $U$ is $g$ .
\end{definition}

\begin{remark}\label{rem:noninjw-lbl}
    When $\phi$ is not an injective wreath recursion, the label for a given element in $ V_{\phi}(G)$ may not be uniquely defined. In fact, consider the wreath recursion in Example \ref{ex:ninj-wr}, where $\phi(g) = ((1,1),\mathrm{id})$. Then for any $g,h\in G$, $\begin{scriptsize}\begin{pmatrix}  
u & u_2 & \cdots & u_n \\
g & g_2 & \cdots &g_n \\
v& v_2 & \cdots &v_n
\end{pmatrix}  \end{scriptsize} $ and $\begin{scriptsize}\begin{pmatrix}  
u & u_2 & \cdots & u_n \\
h & g_2 & \cdots &g_n \\
v& v_2 & \cdots &v_n
\end{pmatrix}  \end{scriptsize} $ represents the same element in $V_\phi(G)$. In particular, the label is not well defined at $u$. 
\end{remark}

%Note that the label depends on the choice of dyadic intervals and even if the label at a dyadic interval is not \s{1_G}, when we restrict to a subinterval it might be \s{1_G}. Therefore the label at a point is not well-defined unless the label is \s{1_G} at some neighborhood of this point. 

Note that the label depends heavily on the choice of dyadic intervals and even if the label at a dyadic interval is not \s{1_G}, when we restrict to a dyadic sub-interval it might be \s{1_G}. On the other hand, if the label at a dyadic interval $U$ is $1_G$, then restricted to any dyadic sub-interval of $U$, it must also be $1_G$. 

\begin{example}
    Consider the right wreath recursion: \s{\phi_r\colon g\ra{((1, g),\mathrm{id}_{\text{S}_2})}}. Given $g_0\in G$, let \s{\Lambda_u(g_0)= \begin{scriptsize}\begin{bmatrix}  
u & u_2 & \cdots & u_n \\
g_0 & 1 & \cdots &1 \\
u& u_2 & \cdots &u_n
\end{bmatrix}  \end{scriptsize}} where \s{\{u, u_2,  \cdots , u_n\}} is a partition set. Then for any dyadic interval $W$ not containing the point ${u}{11}\cdots$, we have $l(_{W}|\Lambda_u(g_0))=1_G$ .
\end{example}

 %Recall from \cref{vg right action} that \s{\alpha=\begin{scriptsize}
 %   \begin{bmatrix}
 %       u&...\\
 %       g&...\\
  %      v&...   
  %  \end{bmatrix}
%\end{scriptsize}} acts on \s{\xinf} by mapping \s{uw} to \s{v(w)\phi^*(g)} for all \s{w\in\xinf}, and we use \s{(w)\alpha} to denote the image of $w$ under the action of $\alpha$. In particular the image $U\alpha$ is again a dyadic interval. Given two elements $\alpha,\beta$, if $l(\alpha|_U) =g$ and $l(\beta|_{U\alpha}) =h$, we have \s{l(\alpha\beta|_U)=gh}. 

 Recall from \cref{vg right action} that $\alpha=\begin{scriptsize}
    \begin{bmatrix}
        u&...\\
        g&...\\
        v&...   
    \end{bmatrix}
\end{scriptsize}\in V_{\phi}(G)$ acts on \s{\xinf} by mapping \s{uw} to \s{v(w)\phi^*(g)} for all \s{w\in\xinf}. In particular $\alpha$ maps the dyadic interval $U = u\xinf$ to another dyadic interval $(U)\alpha=v\xinf$ . Given two elements $\alpha,\beta$ , if $l(_{U}|\alpha) =g$ and $l(_{U\alpha}|\beta) =h$, then $l(_U|\alpha\beta)=gh$. 
%\begin{definition}
 %   The \emph{labeled support} of \s{\alpha \in V_\phi(G)} denoted by %\s{\mathrm{lsupp}(\alpha)} is defined as the complement set of the union of all dyadic neighborhoods \s{U} satisfying \s{l(\alpha|_U)=1_G} and the action of \s{\alpha} restricting to \s{U} is the identity map.
%\end{definition} 
\begin{definition}
    The \emph{labeled support} of \s{\alpha \in V_\phi(G)} denoted by \s{\mathrm{lsupp}(\alpha)} is defined as the complement set of the union of all dyadic intervals \s{U} in all $G$-matrices
    $$\mathcal{M}=\begin{scriptsize}\begin{pmatrix}  
u & u_2 & \cdots & u_n \\
g & g_2 & \cdots &g_n \\
u& v_2 & \cdots &v_n
\end{pmatrix}  \end{scriptsize} $$
    with $[\mathcal{M}]=\alpha$ and $g=1_G$ (\ie\ \ $\alpha$ satisfies $l(_U|\alpha)=1_G$ and $\alpha$ maps $U$ to $U$ identically).
\end{definition} 
\begin{remark}
\begin{enumerate}
     \item Since dyadic intervals are open, their union is also open, one sees that the labeled support of an element is closed. But in general, it does not need to be open; see \Cref{ex:lsupp-notop}(\ref{caser}).
    \item The usual support of $\alpha$ defined using its action on $\xinf$ is contained in $\mathrm{lsupp}(\alpha)$. 
    \item When $\phi$ is not an injective wreath recursion, the labeled support of an element is not well defined since its label at a dyadic interval is not well-defined; see Remark \ref{rem:noninjw-lbl}.
\end{enumerate}
\end{remark}

 Note that the labeled support of an element strongly depends on the wreath recursion.

%\begin{example}
%    Again consider \s{\Lambda_u(g_0)= \begin{scriptsize}\begin{bmatrix}  
%u & u_2 & \cdots & u_n \\
%g_0 & 1 & \cdots &1 \\
%u& u_2 & \cdots &u_n
%\end{bmatrix}  \end{scriptsize}}. For the diagonal wreath recursion, \s{\Delta: g\ra{((g, g),id_{\text{S}_2})}}, $\mathrm{lsupp}(\Lambda_u(g_0)) = u\xinf$. But for the right wreath recursion, \s{\phi_r: g\ra{((1, g),id_{\text{S}_2})}}, $\mathrm{lsupp}(\Lambda_u(g_0)) = \{{u}{11}\cdots\}$ which is closed but not open.
%\end{example}

\begin{example}\label{ex:lsupp-notop}
    Consider again $\Lambda_u(g_0)$. 
    
    \begin{enumerate}
        \item For the diagonal wreath recursion, $$\Delta\colon g\ra{((g, g),\mathrm{id}_{\text{S}_2})}\implies\mathrm{lsupp}(\Lambda_u(g_0)) = u\xinf.$$
        The action of $\Lambda_u(g_0)$ on $u\xinf$ in this case is trivial.
        
        \item \label{caser} For the right wreath recursion, $$\phi_r\colon g\ra{((1, g),\mathrm{id}_{\text{S}_2})}\implies\mathrm{lsupp}(\Lambda_u(g_0)) = \{{u}{11}\cdots\}$$ which is closed but not open.
        The action of $\Lambda_u(g_0)$ in this case is again trivial.

        \item Consider the wreath recursion corresponding to the adding machine in Remark \ref{rem:add-mach}, where $G = \langle t\rangle \cong \mathbb{Z}$, $\sigma$ is the permutation of $\{0,1\}$, and $$\phi(t) = ((1,t),\sigma).$$
        Then $\mathrm{lsupp}(\Lambda_u(t) )= u\xinf$. For convenience, let $u$ be the empty word, then the action of $t$ on $\xinf$ corresponds to adding 1 to a dyadic integer; see \cite[1.7.1]{Nek05} for more details. 
    \end{enumerate}
\end{example}

The labeled support of a subset of $V_\phi(G)$ is defined as the union of the labeled support of all elements in this set.

\begin{lemma}\label{commute}
Two elements with disjoint labeled supports commute.
\end{lemma}
\begin{proof}
  Given two elements \s{\alpha,\beta} with disjoint labeled supports, then after some expansions they can be written as $$
\alpha=\begin{scriptsize}\begin{bmatrix}   
    u_1&\cdots &u_k&v_1&\cdots &v_{k'}\\
    1_G&\cdots&1_G&g_1&\cdots&g_{k'}\\
    u_1&\cdots &u_k&w_1&\cdots &w_{k'}
  \end{bmatrix}\end{scriptsize}, 
\beta=\begin{scriptsize}
\begin{bmatrix}  w_1&\cdots &w_{k'}&u_1&\cdots &u_{k}\\
1_G&\cdots&1_G&h_1&\cdots&h_{k}\\
w_1&\cdots &w_{k'}&t_1&\cdots &t_{k}
\end{bmatrix}\end{scriptsize},
$$  
where these \s{ u_i\xinf} together with \s{v_j\xinf} form a dyadic partition of \s{\mc}.

So $$\alpha\beta=\begin{scriptsize}
    \begin{bmatrix}
        u_1&\cdots&u_k&v_1&\cdots&v_{k'}\\
        h_1&\cdots&h_k&g_1&\cdots&g_{k'}\\
        t_1&\cdots&t_k&w_1&\cdots&w_{k'}
        
    \end{bmatrix}
\end{scriptsize}.$$

Note we can also write

$$
\alpha=\begin{scriptsize}\begin{bmatrix}   
    t_1&\cdots &t_k&v_1&\cdots &v_{k'}\\
    1_G&\cdots&1_G&g_1&\cdots&g_{k'}\\
    t_1&\cdots &t_k&w_1&\cdots &w_{k'}
  \end{bmatrix}\end{scriptsize},
\beta=\begin{scriptsize}
\begin{bmatrix}  v_1&\cdots &v_{k'}&u_1&\cdots &u_{k}\\
1_G&\cdots&1_G&h_1&\cdots&h_{k}\\
v_1&\cdots &v_{k'}&t_1&\cdots &t_{k}
\end{bmatrix}\end{scriptsize}.
$$  

One checks easily now that $\alpha \beta =\beta \alpha$.
\end{proof}

\begin{definition}
Let $\omega_0\stackrel{\text{def}}{=}000\cdots$ be the constant \s{0} word in \s{\mc}.
Two elements \s{\alpha_1, \alpha_2\in \vg} are called \emph{equivalent} if
there is a dyadic neighborhood $U$ of \s{\omega_0} such that 
 \s{(w)\alpha_1=(w) \alpha_2} for all \s{w\in U} and
$l(_U|\alpha_1)= l(_U|\alpha_2)$ . 
A \emph{labeled germ} is an equivalence class of $V_{\phi}(G)$ under this relation. For any $\alpha\in V_{\phi}(G)$ , we use \s{[\alpha]} to denote its equivalence class. 
Denote the set of labeled germs  by \s{\germ}. 
\end{definition}

Define a right action of \s{\vg}  on \s{\germ} by $[\alpha] \beta\stackrel{\text{def}}{=}[\alpha \beta]$.
This is a well-defined action: assume \s{\alpha_1} equivalent to \s{\alpha_2}, let the corresponding dyadic neighborhood $U$ small enough so that the label of \s{\beta} at $U\alpha_1=U\alpha_2$ is also well-defined, then \s{(w)\alpha_1\beta=(w)\alpha_2\beta} for all $w\in U$ and 
$$l(_{U}|\alpha_1\beta)=l(_{U}|\alpha_1)l(_{U\alpha_1}|\beta)=l(_{U}|\alpha_2)l(_{U\alpha_2}|\beta)=l(_{U}|\alpha_2\beta).$$

We proceed to define a binary relation $\perp$ on $\germ$.
\begin{definition}
    Given two labeled germs \s{[\alpha_1], [\alpha_2]},  we say \s{[\alpha_1]\perp[\alpha_2]} if \s{(\omega_0)\alpha_1\neq (\omega_0)\alpha_2}.
\end{definition}

\begin{remark}
    If \s{(\omega_0)\alpha_1\neq (\omega_0)\alpha_2}, we obviously have $[\alpha_1]\neq [\alpha_2]$. But the converse is not necessarily true.
\end{remark}

\begin{lemma}
   ``$\perp$" defines a generic relation on \s{\germ} and \s{\vg} preserves the  relation.  
\end{lemma}
\begin{proof}
    The action of \s{\vg} obviously preserves the relation. To show that  ``$\perp$" defines a generic relation on \s{\germ} it suffices to observe that the orbit of the element $\omega_0$ under the action of \s{\vg} on $\{0,1\}^\omega$ has infinitely many elements.
\end{proof}

Let \s{\germ^\perp_n} denote the set of generic \s{(n+1)}-tuples of \s{\germ}. Each element $\gamma$ of the Thompson group $V$ is regarded as an element of $V_{\phi}(G)$ with trivial labels, then its labeled support equals to its support as a homeomorphism on the Cantor set.

\begin{lemma}\label{vigorous}   \cite[Section 1.1]{Co08}
For any clopen sets \s{A,B,C} in the Cantor set \s{\mc} with \s{B,C} properly contained in \s{A}, there exists a \s{\gamma\in V} such that \s{\mathrm{supp}(\gamma)\subset A} and \s{B\gamma  \subset C}.
\end{lemma}

\begin{lemma}\label{Lemma:vg-transitive}
    \s{\vg} acts transitively on \s{\germ^\perp_n} for all \s{n\geq 0}. 
\end{lemma}
\begin{proof}
    Given two generic tuples \s{([\alpha_0],  \cdots , [\alpha_n])}, \s{([\beta_0],  \cdots , [\beta_n])\in \germ^\perp_n}, we want to find an element in  \s{\vg} that maps $[\beta_i]$ to $[\alpha_i]$. Note first that \s{(\omega_0)\alpha_0,  \cdots , (\omega_0)\alpha_n} are  pairwise distinct by definition, and the same holds for \s{(\omega_0)\beta_0,  \cdots , (\omega_0)\beta_n}. Choose a small dyadic neighborhood $U$ of \s{\omega_0} such that both finite sets \s{ \{U \alpha_0,  \cdots ,  U \alpha_n\}} and \s{ \{U \beta_0,  \cdots ,  U \beta_n\}}  are  pairwise disjoint dyadic intervals. Then there is an element \s{\gamma\in \vg} such that \s{l(_{ U \beta_i}|\gamma)=l(_U|\beta_i)^{-1}l(_U|\alpha_i)} and \s{(w)\beta_i\gamma=(w)\alpha_i} for all \s{w\in U}. In more detail, we can find two dyadic partitions $\{u_0,\cdots,u_k\}$ and $\{v_0,\cdots,v_k\}$  for some $k\geq n$ such that $u_i\xinf = U\alpha_i$  and $v_i\xinf = U\beta_i$   for $0\leq i\leq n$.
Let 
$$\gamma \stackrel{\text{def}}{=}\begin{scriptsize}
     \begin{bmatrix}
         v_0&\cdots &v_n&v_{n+1}&\cdots&v_k\\
         h_1&\cdots&h_n&1_G&\cdots&1_G\\
         u_0&\cdots &u_n&u_{n+1}&\cdots&u_k
     \end{bmatrix}
 \end{scriptsize}$$
 where $h_i\stackrel{\text{def}}{=}l(_U|\beta_i)^{-1}l(_U|\alpha_i)$ . It's easy to check that $[\beta_i]\curvearrowleft\gamma$ is equal to $[\alpha_i]$.
\end{proof}

\begin{lemma}\label{Lemma:l_thompson-Stab-ba}
    The stabilizers for the \s{\vg}-action on \s{\germ^\perp_n} are boundedly acyclic for all \s{n\geq 0}. 
\end{lemma}
\begin{proof}
  
  Let \s{([x_0],  \cdots [x_n])\in \germ^\perp_n}  where each \s{(\omega_0)x_i} belongs to \s{00\xinf} for all \s{0\leq i\leq n}. Since the action of \s{\vg} on \s{\germ^\perp_n} is transitive, it suffices to prove that the stabilizer \s{G_n} of \s{([x_0],  \cdots [x_n])} is boundedly acyclic. By definition, it's easy to see that \s{G_e=\{\gamma\in \vg:\omega_0\notin \mathrm{lsupp}(\gamma)\}} where \s{G_e=\mathrm{Stab}_{\vg}([e])} and $e$ is the identity element in $V_{\phi}(G)$. Hence we have \s{G_n=\bigcap_{i=0}^{n} G_e^{x_i}=\{\gamma\in \vg:(\omega_0)x_i\notin \mathrm{lsupp}(\gamma) \text{ for each } 0\leq i\leq n\}}.  Now let \s{\Omega=10\xinf}, 
  denote by \s{G_\Omega} the subgroup consisting of elements with labeled support in \s{\Omega}, so \s{G_{\Omega}\leq G_n}. We claim the pair \s{(G_n,G_\Omega)} satisfies the conditions in \cref{coame-to-wreath}, hence \s{G_n} is boundedly acyclic.

First, we claim that every finite subset $\text{F}$ of \s{G_n}
   can be conjugated by an element of \s{V\cap G_n} into \s{G_{\Omega}}.
  In fact, \s{\mathrm{lsupp}(\text{F})} is a  closed hence proper subset of \s{\mc-\{(\omega_0)x_0,\cdots ,(\omega_0)x_n\}}, so we could choose clopen subsets $A,B$ of $\mc-\{(\omega_0)x_0,\cdots ,(\omega_0)x_n\}$ such that $\mathrm{lsupp}(\text{F}) \subsetneq B\subsetneq A$ and $\Omega \subsetneq A$. Now by \cref{vigorous}  we can pick a \s{\gamma_1\in V\cap G_n} such that \s{B\gamma_1\subset \Omega  }, hence $\mathrm{lsupp}(\text{F})\gamma_1\subset \Omega$  and \s{\text{F}^{\gamma_1}\subset G_{\Omega }}. 

Next, we  define an element \s{\gamma_2\in V} as follows:
   $$ (uw)\gamma_2  =\begin{cases}
    00w   &  \text{ if } u=00\\
    010w  &  \text{ if } u=01\\
    011w   &  \text{ if } u=10\\
    1w   &  \text{ if } u=11\\
\end{cases}  $$   
Note that  \s{\gamma_2\in V\cap G_n} and \s{  \Omega\gamma_2^k \cap \Omega =\varnothing} for all \s{k\in\mz-\{0\}}. Therefore, by \cref{commute} the conjugation subgroup of \s{G_\Omega} by \s{\gamma_2^p} and by \s{\gamma_2^q} always commute if \s{p\neq q}. 
   \end{proof}

\begin{remark}
    One could also use \Cref{commuting conjugates} to prove \Cref{Lemma:l_thompson-Stab-ba}. 
\end{remark}

Now combining \Cref{thm:inj-iso}, \Cref{Lemma:vg-transitive} and \Cref{Lemma:l_thompson-Stab-ba} with \cref{cor:BA-gene} we have: 
\begin{thm}\label{Thm:l-Thompson-ba}
    \s{\vg} is boundedly acyclic for any  wreath recursion $\phi$. 
\end{thm}

\section{Bounded cohomology of some topological full groups}
In this section, we  prove  topological full groups with extremely proximal actions are boundedly acyclic. 
Recall given a subgroup \s{\Gamma\subseteq \operatorname{Homeo}(\mc)}, its \emph{topological full group} is defined as:$$\llbracket \Gamma \rrbracket=\left\{g \in \operatorname{Homeo}(\mc): \begin{array}{l}
\text { for each } x \in \mc \text { there exists a neighbour- } \\
\text { hood } U \text { of } x \text { and } \gamma \in \Gamma \text { such that } ~_U| g=~_U|\gamma
\end{array}\right\}$$

Note by definition, we have $\llbracket \llbracket\Gamma \rrbracket \rrbracket = \llbracket\Gamma \rrbracket $ and a subgroup $\Gamma\subset \operatorname{Homeo}(\mc)$ is a topological full group if and only if $\llbracket\Gamma\rrbracket=\Gamma$.

\begin{definition}
    Let $\Gamma$ be a group acts on the Cantor set, the action is \emph{extremely proximal} if for any nonempty and proper clopen sets \s{V_1,V_2\subsetneq \mc} there exists \s{\gamma\in \Gamma} such that \s{\gamma(V_2)\subsetneq V_1}. 
\end{definition}

Inspired by the uniform simplicity results of topological full groups in \cite[Section 5]{GG17}, in this section we will prove the following.

\begin{thm}\label{thm:tofull-ba}
    Let $\Gamma \leq \mathrm{Homeo}(\mc)$ be a topological full group  such that the action is extremely proximal. Then \s{\Gamma} is boundedly acyclic.    
\end{thm}
\begin{proof}
    First let \s{p\in\mc}, we define \s{X_\Gamma:=\Gamma/\sim} where \s{\gamma_1\sim \gamma_2} when they agree on a neighborhood of \s{p}, let \s{[\gamma]\in X_\Gamma} denote the equivalence class of \s{\gamma}. View \s{X_\Gamma} as a  \s{\Gamma}-set where \s{\Gamma} acts by \s{[\gamma_1]\gamma_2:=[\gamma_1\gamma_2]}. Then we define a relation in \s{X_\Gamma} by setting \s{[\gamma_1]\perp[\gamma_2]} if \s{(p)\gamma_1\neq (p)\gamma_2}. Now it's  easy to see that the right \s{\Gamma} action preserves the relation \s{\perp} and \s{\perp} is actually generic in the sense of \cref{generic}. So we get a boundedly acyclic \s{\Gamma}-semi-simplicial set \s{(X_\Gamma)_\bullet^\perp} by \Cref{lemma:generic-ba}. For disjoint clopen sets \s{A, B} and \s{\tau\in\Gamma} maping \s{A} to \s{B} homeomorphically, define a homeomorphism $$(x)\gamma_{\tau, A,B}:=\begin{cases}
        (x) \tau,&x\in A\\
        (x)\tau^{-1},&x\in B\\
        x, &otherwise
    \end{cases}$$ 
    By definition \s{\gamma_{\tau, A,B}\in \Gamma} and is of order $2$.

We want to apply \cref{cor:BA-gene} to show $\Gamma$ is boundedly acyclic. For that  we first prove that for each \s{n\geq 0} the \s{\Gamma-} action on the set \s{(X_\Gamma)_n^\perp} is transitive. Let us prove it by induction on \s{n}. For \s{n=0} it is trivial, since $(X_\Gamma)_0^\perp=X_\Gamma$ which is $\Gamma$-transitive by definition. Assume that  the actions are transitive up to $n-1$. Given \s{([\gamma_0],\cdots,[\gamma_n]), ([\gamma'_0],\cdots,[\gamma'_n])\in (X_\Gamma)_n^\perp}, by the induction hypothesis there is \s{\gamma\in\Gamma} with \s{[\gamma_i]\gamma=[\gamma'_i]} for \s{i<n}. We first claim that \s{\gamma} can be chosen such that \s{(p)\gamma_n\gamma\neq (p)\gamma_n'}. In fact, if \s{(p)\gamma_n\gamma= (p)\gamma_n'}, then we can modify \s{\gamma} in the following way.  Choose $A_1$ to be a 
 small clopen subset disjoint from $(p)\gamma_0\gamma,\cdots,(p)\gamma_{n-1}\gamma$ but containing $(p)\gamma_n\gamma$, and $B_1$ to be another clopen subset disjoint from $(p)\gamma_0\gamma,\cdots,(p)\gamma_{n-1}\gamma$ and $A_1$. If we replace \s{\gamma} with \s{\gamma\gamma_{\tau_1, A_1, B_1}}, then  \s{(p)\gamma_n\gamma\neq (p)\gamma_n'} and \s{[\gamma_i]\gamma=[\gamma'_i]} for \s{i<n}. Now let  \s{\alpha:=\gamma_n\gamma}, \s{\beta:=\alpha^{-1}\gamma'_n} , they are elements of \s{\Gamma}. Note that \s{(p)\alpha \neq (p)\alpha\beta}. Pick a small clopen set \s{A_2} containing \s{(p)\alpha}  but disjoint from \s{(p)\gamma'_0,\cdots,(p)\gamma'_{n-1}}, such that its image $B_2 = A_2 \beta$ is also disjoint from \s{(p)\gamma'_0,\cdots,(p)\gamma'_{n-1}}. Note now that the element \s{\gamma_{\beta, A_2, B_2}\in \Gamma}, fixes \s{\{(p)\gamma'_0,\cdots,(p)\gamma'_{n-1}\}} pointwise and maps \s{(p)\alpha} to \s{(p)\alpha\beta}.
One  checks also that \s{([\gamma_0],\cdots,[\gamma_n])\gamma\gamma_{\beta, A_2, B_2}= ([\gamma'_0],\cdots,[\gamma'_n])}.

    Next we prove that for each \s{n\geq 0} the stabilizers of \s{\Gamma}-action on \s{(X_\Gamma)_n^\perp} is boundedly acyclic. Let \s{([\gamma_0],\cdots,[\gamma_n])\in (X_\Gamma)_n^\perp} and \s{\Gamma_n} be its stabilizer, let \s{F:=\{(p)\gamma_0,\cdots,(p)\gamma_n\}}. Then \s{\Gamma_n} consists of homeomorphisms that fix a neighborhood of \s{F} pointwise. We will prove that \s{\Gamma_n} is boundedly acyclic by showing that it  has commuting cyclic conjugates (see \cref{def of cc}). For a given finitely generated group \s{H\leq \Gamma_n}, there is a neighborhood \s{U} of \s{F} that is fixed by every element of \s{H} pointwise. Pick a clopen set \s{U'} with \s{F\subsetneq U'\subsetneq U}. Since the \s{\Gamma}-action on \s{\mc} is extremely proximal, there is \s{ \eta\in \Gamma} such that \s{(\mc-U)\eta\subset U-U'}. 
   Note that \s{\gamma_{\eta,\mc-U,(\mc-U)\eta}\in \Gamma_n} and is of order 2, and \s{H} commutes with its conjugate \s{H^{\gamma_{\eta,\mc-U,(\mc-U)\eta}}} since they have disjoint support. Therefore  \s{\Gamma_n} has the commuting cyclic conjugates for \s{k=2} in \cref{def of cc}.  Hence by \cref{commuting conjugates}, \s{\Gamma_n} is boundedly acyclic. This finishes the proof.
\end{proof}

    In \cite[Section 5]{GG17}, Gal and Gismatullin proved that topological full groups acting on $\mc$  extremely proximally are uniformly simply. Recall that a group is \emph{uniformly simple} if for any given conjugacy class $C$, every element of the group can be written as the product of bounded number of elements from $C^{\pm}$. They checked that the commutator subgroup of the \emph{Neretin groups} and the commutator subgroups of  the \emph{Higman--Thompson groups} satisfy these assumptions, hence they are also boundedly acyclic by our theorem. But a group is always coamenable with its commutator subgroup, hence  we have the following corollary by \Cref{Lem:coame-bcohg}. 

\begin{corollary}
     Neretin groups  (viewed as discrete groups) and  Higman--Thompson groups are boundedly acyclic.
\end{corollary}

It is also the case that R\"over--Nekrashevych groups are topological full groups (see for example \cite[Example 5.18]{BFG24}) and act on the Cantor set extremely
proximally since the Higman--Thompson groups already do. 

\begin{corollary} \label{cor:RNGroup-bc}
    The R\"over--Nekrashevych groups are all boundedly acyclic.
\end{corollary}

    Note that for simplicity, we only dealt with $V_\phi(G)$ in the previous sections, so Theorem \ref{Thm:l-Thompson-ba} only works for R\"over--Nekrashevych groups in the degree $2$ case. Another class of groups satisfies the assumptions of our theorem is the twisted Brin--Thompson groups of Belk and Zaremsky \cite{BZ22}. In fact, it is known that they are topological full groups, and act extremely proximally on the Cantor space, see the proof of \cite[Corollary 5.14, Proposition 5.23]{BFG24}. 

\begin{corollary}\label{cor:tbt-bc}
    The twisted Brin--Thompson groups  are all  boundedly acyclic.
\end{corollary}

\begin{remark}
     Gal and Gismatullin also proved the uniform simplicity for commutator subgroups of groups acting proximally and boundedly on the real line $\mathbb{R}$ in \cite[Theorem 3.1]{GG17}. Recall that an action of a group $G$ on $\mathbb{R}$ is proximal if for every $a,b,c,d \in \mathbb{R}$, such that $a < b$ and $c < d$ there is $g\in G$ satisfying
$g(a,b) \supseteq  (c,d)$; it is bounded if every element has bounded support. Note that proximal actions must have no global fixed point. Hence by \cite[Corollary 1.13]{CFLM23}, such groups are also boundedly acyclic. 
\end{remark}

    In general, topological full groups could have non-trivial bounded cohomology. 

    \begin{example}
         Let $G$ be a countable group that acts on the Cantor set $\{0,1\}^{G}$ by permuting the coordinates using the right multiplication. Then the corresponding  topological full group $\llbracket G\rrbracket$ retracts to $G$. In fact, let  $w_0 = \cdots 000\cdots$ 
 be the constant $0$ sequence in $\{0,1\}^{G}$, then it is a fixed point of $G$ hence also $\llbracket G\rrbracket$. Note that the clopen subsets of the form $\{w\in \{0,1\}^{G} \mid w_{g} =a_g, g\in F\} $ where $F$ is a finite subset of $G$ and $a_g \in \{0,1\}$  form a topological basis of $\{0,1\}^{G}$. In particular, given an element of $\llbracket G\rrbracket \to G$, one can find such a clopen neighborhood $U_0$ of $w_0$, such that the action restricted to $U_0$ determines an element of $G$. One further checks that this element does not depend on the choice of $U_0$. Thus we have a well-defined retraction map  $r:\llbracket G\rrbracket \to G$. In fact, let $\iota$ be the natural inclusion map from $G$ to  $\llbracket G\rrbracket$, we have $r\circ \iota = \mathrm{id}_G$. Hence  $r$ induces a splitting injective map on bounded cohomology. In particular, $\llbracket G\rrbracket$ is not boundedly acyclic when $G$ is not.
    \end{example}

\section{$ \ell^{2}$-invisibility}\label{sec:l2-inv}
Let $\Gamma$ be a discrete group and $\mathcal{N}(\Gamma)$ its von Neumann algebra. We prove in this section that our embedding results also work in the setting of $\ell^2$-invisibility. Recall that $\Gamma$ is called \emph{$\ell^{2}$ invisible} if $H_{i}(\Gamma ; \mathcal{N}(\Gamma))=0$ for all $i \geq 0$. If $\Gamma$ is of type $F_{\infty}$, this is equivalent to $H_{i}\left(\Gamma ; \ell^{2}(\Gamma)\right)=0$ for all $i \geq 0$ \cite[Lemmas 6.98 and 12.3]{Luc02}.
Note that $\ell^{2}$-invisibility of a group is a much stronger property than the vanishing of
its $ \ell^{2}$-Betti numbers.

A major open problem, strongly related to the zero in the spectrum conjecture, is whether there exist $\ell^{2}$-invisible groups of type $F$ \cite[Remark 12.4]{Luc02}. Lück here also asked about the existence of groups of type $F_{\infty}$ that are $\ell^{2}$-invisible. For the case of groups of type $F_{\infty}$, the first examples were provided by Sauer and Thumann in \cite{ST14}, which was later generalized by Thumann in \cite{Thu16} (finitely presented examples of a different nature, which are not of type $F_{\infty}$, were constructed by Oguni \cite{Ogu07}). The Sauer--Thumann paper covers many groups, but it seems that $V$ was the first example.

It turns out that the labeled Thompson group $V(G)$ (even the $F$ counterparts) are always $\ell^{2}$-invisible, as soon as the input group $G$ is non-amenable. 

\begin{thm}\label{l2-invi}
    If $G$ is non-amenable, then $V(G)$, $F(G)$ and $T(G)$ are $\ell^{2}$-invisible.
\end{thm} 
This is just a combination of known facts about $\ell^{2}$-invisibility.
\begin{lemma}\cite[Lemma 12.11(1)]{Luc02}\label{Lemma:ld-uni}
    Let $\Gamma$ be a directed union of the groups $\Gamma_{\alpha}$. If $H_{i}\left(\Gamma_{\alpha} ; \mathcal{N}\left(\Gamma_{\alpha}\right)\right)=0$ for all $i=0, \ldots, n$, then $H_{i}(\Gamma ; \mathcal{N}(\Gamma))=0$ for all $i=0, \ldots, n$.
\end{lemma}  

\begin{lemma}\cite[Lemma 12.11(2)]{Luc02}\label{Lemma:ld-norsub}
     Let $\Lambda \leq \Gamma$ be a normal subgroup such that $H_{i}(\Lambda ; \mathcal{N}(\Lambda))=0$ for all $i \leq n$. Then $H_{i}(\Gamma ; \mathcal{N}(\Gamma))=0$ for all $i \leq n$.
\end{lemma}

\begin{lemma}\cite[Lemma 12.11(3)]{Luc02}\label{Lemma:ld-prd}
Let $\Gamma_{1}, \Gamma_{2}$ be such that $H_{i}\left(\Gamma_{j} ; \mathcal{N}\left(\Gamma_{j}\right)\right)=0$ for $i \leq n_{j}$. Then $H_{i}\left(\Gamma_{1} \times \Gamma_{2} ; \mathcal{N}\left(\Gamma_{1} \times \Gamma_{2}\right)\right)=0$ for $i \leq n_{1}+n_{2}+1$.
\end{lemma}

\begin{lemma}\cite[Lemma 6.36]{Luc02}\label{Lemma:ld-non-ame}
$\Gamma$ is non-amenable if and only if $H_{0}(\Gamma ; \mathcal{N}(\Gamma))=0$.
\end{lemma}

\begin{proof}[Proof of \cref{l2-invi}]
Recall that at the end of Section \ref{sec:V(G)}, we have a short exact sequence for $V(G)$:
\[1 \rightarrow K \rightarrow V(G) \rightarrow V \rightarrow 1,\] 
where the kernel $K$ can be described by the following direct limit: $G_n = \oplus_{2^n} G$, and injection maps $G_n \to G_{n+1}$ are given by $(g_1,g_2\cdots g_{2^n}) \to (g_1,g_1,g_2,g_2,\cdots, g_{2^n}, g_{2^n})$.
By \cref{Lemma:ld-norsub}, it suffices to prove that $K$ is $\ell^{2}$-invisible. Since $G$ is non-amenable, combining \cref{Lemma:ld-prd} and \cref{Lemma:ld-non-ame}, we see that $H_{i}\left(\oplus_n G ; \mathcal{N}\left(\oplus_n G\right)\right)=0$ for all $i<n$. By \cref{Lemma:ld-uni} we conclude that $K$ is $\ell^{2}$-invisible. The exact same proof also shows that $F(G)$ and $T(G)$ are $\ell^{2}$-invisible.
\end{proof}

\begin{remark}
    Let $G$ act on an infinite set $S$, and let $H$ be a non-amenable group. The above proof also shows that any wreath product of the form $H\wr_S G$ is $\ell^{2}$-invisible. One could also use wreath products to embed groups into boundedly acyclic groups with  $\mathfrak{X}^{\text {ssep }}$ coefficients, see \cite[Remark 8.3]{Bo25} for more information. 
\end{remark}

\begin{corollary}\label{cor-l2-emb}
    Every group of type $F_n$ embeds  quasi-isometrically into an $\ell^{2}$-invisible group of type $F_n$ that has no proper finite index subgroup. 
\end{corollary}

\begin{proof}
    Let $G$ be a group of type $F_n$, then $G$ embeds quasi-isometrically in $V(G\ast \mathbb{Z})$, hence in $V(G\ast \mathbb{Z})$ by \Cref{quasi-retract}. Since $G\ast \mathbb{Z}$ is non-amenable, we have $V(G\ast \mathbb{Z})$ is $\ell^{2}$-invisible. By \Cref{cor-VG-perfect}, $V(G\ast \mathbb{Z})$ has no proper finite index subgroup. By \Cref{thm:fin-VG}, $V(G\ast \mathbb{Z})$ is of type $F_n$.
\end{proof}

\end{document}